\documentclass[12pt, reqno]{amsart}
\usepackage{amsmath,amssymb,amsthm}

\usepackage{amsmath, amssymb}
\usepackage{amsthm, amsfonts, mathrsfs}
\usepackage{mathptmx}
\usepackage{fullpage}
\usepackage{amsfonts,graphicx}
\numberwithin{equation}{section}
\usepackage[colorlinks=true, pdfstartview=FitV, linkcolor=blue, citecolor=blue, urlcolor=blue]{hyperref}

\newcommand{\q}{{\mathbb{Q}}}
\newcommand{\p}{{\mathbb{P}}}
\newcommand{\R}{{\mathbb{R}}}
\def\div{ \hbox{\rm div}\,  }

\newcommand{\Z}{{\mathbb{Z}}}

\def\ta{\varphi}
\def\La{ \Lambda }
\def\nn{\nonumber}

\def\R{{\mathbb R}}
\def\div{ \hbox{\rm div}\,  }
\def\Z{{\mathbb{Z}}}
\def\u{\mathbf{u}}
\def\v{\mathbf{v}}
\def\w{\mathbf{w}}

\def\La{\Lambda}

\def\ddj{\dot \Delta_j}

\def\e{\mathcal{E}_\infty}
\def\ee{\mathcal{E}_1}
\def\er{\mathcal{D}_1(t)}
\def\es{\mathcal{D}_\infty(t)}
\newcommand{\norm}[1]{\lVert #1 \rVert}
\theoremstyle{plain}

\newtheorem{theorem}{Theorem}[section]
\newtheorem{lemma}[theorem]{Lemma}
\newtheorem{definition}[theorem]{Definition}

\newtheorem{proposition}[theorem]{Proposition}
\newtheorem{remark}[theorem]{Remark}
\numberwithin{equation}{section}

\begin{document}
\title[Viscous and non-resistive MHD equations ]{Global small solutions to a special $2\frac12$-D compressible viscous non-resistive MHD system}

\author[B. Dong]{Boqing Dong}
\address[B. Dong]{School  of Mathematics and Statistics, Shenzhen University,
 Shenzhen, 518060, China.} \email{bqdong@szu.edu.cn}
\author[J. Wu]{Jiahong Wu}
\address[J. Wu]{Department of Mathematics, Oklahoma State University, 401 Mathematical Sciences, Stillwater,
OK 74078, USA. } \email{jiahong.wu@okstate.edu}
 \author[X. Zhai]{Xiaoping Zhai}
\address[X. Zhai]{School of Mathematics and Statistics, Guangdong University of Technology,
Guangzhou, 510520, China.} \email{zhaixp@szu.edu.cn (Corresponding author)}

\subjclass[2020]{35Q35, 35A01, 35A02, 76W05}

\keywords{Global solutions; Non-resistive compressible MHD; Decay rates}

\begin{abstract}
This paper solves the global well-posedness and stability problem on a special
$2\frac12$-D compressible viscous non-resistive MHD system near a steady-state solution.
The steady-state here consists of a positive constant density and a background magnetic field.
The global solution is constructed in $L^p$-based homogeneous Besov spaces, which allow general
and highly oscillating initial velocity. The well-posedness problem studied here is extremely
challenging due to the lack of the magnetic diffusion, and remains open for the
corresponding 3D MHD equations. Our approach exploits the enhanced dissipation and
stabilizing effect resulting from the background magnetic field, a phenomenon
observed in physical experiments.  In addition, we obtain the solution's optimal decay rate
when the initial data is further assumed to be in a Besov space of negative index.
	\end{abstract}

\maketitle

\tableofcontents
\section{Introduction and the main results}

The small data global well-posedness problem on the three-dimensional (3D) compressible
viscous non-resistive magnetohydrodynamic (MHD) equations remains an challenging open
problem. Mathematically the concerned MHD equations are given by
\begin{eqnarray}\label{mm1}
	\left\{\begin{aligned}
		&\partial_t\rho+\div(\rho \v)=0,\\
		&\rho(\partial_t\mathbf{v} + \v\cdot \nabla \v) -  \mu \Delta \v -(\lambda+\mu) \nabla \div \v+\nabla P=  (\nabla\times \mathbf{B})\times \mathbf{B},\\
		&\partial_t\mathbf{B}-\nabla\times(\v\times\mathbf{B})=0 ,\\
		&\div\mathbf{B}=0,
	\end{aligned}\right.
\end{eqnarray}
where  $\rho$ denotes the density of the fluid, $\v$  the  velocity field,  and $\mathbf{B}$ the magnetic field. The parameters $\mu$  and $\lambda$ are shear viscosity and volume viscosity coefficients,  respectively,  which satisfy the standard  strong parabolicity assumption,
\begin{align*}
	\mu>0\quad\hbox{and}\quad
	\nu\stackrel{\mathrm{def}}{=}\lambda+2\mu>0.
\end{align*}
The  pressure  $P=A\rho^\gamma$ for some $A>0$ and $\gamma\ge1$. The
compressible MHD equations model the motion of electrically conducting fluids in the presence of
a magnetic field. The compressible MHD  equations  can be derived from the isentropic Navier-Stokes-Maxwell system by taking
the zero dielectric constant limit \cite{lifucai}. When the effect of the magnetic field can be neglected or $\mathbf{B} = 0$,  \eqref{mm1} reduces to the isentropic compressible Navier-Stokes equations.

\vskip .1in
The goal of this paper is to solve the small data global well-posedness problem on a very special
two-and-half-dimensional ($2\frac12$-D) compressible
viscous non-resistive MHD equations (to be specified later). In addition,
we are also interested in the precise large-time behavior of the solutions.

\vskip .1in
Due to its wide physical applications and mathematical challenges, the compressible MHD equations
have attracted the interests of many
 physicists and mathematicians (see, e.g., \cite{biandongfen,Davidson,DJJ,Ducomet,fanjishan,FNS,HCH,HT,wuyifei,zhushengguo,zhongxin}  and the references therein). We briefly recall some results concerning the multi-dimensional barotropic compressible MHD equations, which are closely related to our investigation here.
Ducomet and Feireisl \cite{Ducomet} considered the heat-conducting fluids together with the influence of radiation, and obtained the global existence of weak solutions with finite energy initial data. Hu and Wang \cite{huxianpeng} proved the global existence of weak solutions to the 3D isentropic compressible MHD system via the Lions-Feireisl theory, see \cite{Lions} and \cite{Feireisl}.
We remark that there are  essential differences between the vacuum case and the non-vacuum case.
The global weak solution in the case of vacuum was obtained in the work of Li, Xu and Zhang \cite{lihailiang}. The local well-posedness in the framework of critical Besov spaces was shown
by Bian and Yuan \cite{biandongfen} when there is full dissipation and no vacuum. In the case of
vacuum and no magnetic diffusion, Li, Su and Wang \cite{wangdehua} proved the local
existence and uniqueness of strong solutions.
The small data global well-posedness problem is extremely difficult when there is
no magnetic diffusion.
There are some satisfactory results in the simplified 1D
geometry. Jiang and Zhang \cite{jiang2017} proved the existence and uniqueness of global strong solution
to the isentropic case with large initial data. We refer to \cite{li23}, \cite{li24}, \cite{li25} for more results in 1D concerning isentropic and heat-conductive non-resistive MHD system with large initial data.
 Wu and Wu \cite{wuyifei} presented a systematic approach to the
small data global well-posedness and
stability problem on the 2D compressible non-resistive MHD equations if the initial data close to an equilibrium state, especially with a background magnetic field.
 It appears difficult to extend the approach of \cite{wuyifei}
to $\R^3$.
 There are
some differences between 2D case and 3D case. For 2D case, when
applying $\nabla$ on equations, there will appear at least one good part in nonlinear terms. For
example, $\partial_1\v\cdot\nabla \mathbf{B}=\partial_1\v_1\partial_1 \mathbf{B}+\partial_1\v_2\partial_2 \mathbf{B}$  and $\partial_2\v\cdot\nabla \mathbf{B}$ (coming from $\nabla\v\cdot\nabla \mathbf{B}$) always contain
a strong dissipative part. However, this will not hold for 3D case.
 Tan and Wang \cite{TW} obtained the global existence of smooth solutions to the 3D compressible barotropic viscous non-resistive MHD system in the horizontally infinite flat layer $\Omega=\R^2\times(0,1).$
Initial- and boundary-value problems under some additional compatibility conditions for the 3D compressible MHD equations were examined by Fan and Yu \cite{fanjishan} and local solutions were
obtained even when there is a vacuum. Zhu \cite{zhushengguo} extended the result obtained in  \cite{wangdehua} to the case
of allowing non-negativity of the initial density. We mention that
there are many interesting results on the zero Mach limit results on the incompressible MHD equations
(see, e.g.,\cite{DJJ, HW3, FNS,JJL1, Li, lifucai2}).

\vskip .1in
If we neglect the effect of the magnetic field, the system \eqref{mm1} reduces to the compressible Navier-Stokes equations, which have also been studied by many researchers, see  \cite{danchin2000}, \cite{danchin2014},   \cite{helingbing}, \cite{xujiang}, \cite{xin6}, \cite{xin8}, \cite{xin3}, \cite{xin5}, \cite{xin7}, \cite{xin2}, \cite{zhaixiaoping} and the references therein.

\vskip .1in
Although the small data global well-posedness on the 2D compressible MHD equations
without magnetic diffusion has been successfully settled, this same problem on the 3D counterpart
appears to be inaccessible at this moment. This paper focuses on a very special $2\frac12$-D  compressible MHD system. The motion of fluids takes place in
the plane $\R^2$ while the magnetic field acts on fluids only in the vertical direction, namely
\begin{align*}
&\v=(\v^1(t,x_1,x_2),\v^2(t,x_1,x_2),0)\stackrel{\mathrm{def}}{=}(\u,0),\nn\\
&\rho\stackrel{\mathrm{def}}{=}\rho(t,x_1,x_2),\quad\mathbf{B}\stackrel{\mathrm{def}}{=}(0,0,m(t,x_1,x_2)).
\end{align*}
Then \eqref{mm1} is reduced to
\begin{eqnarray}\label{0mm2}
\left\{\begin{aligned}
&\partial_t\rho+\div(\rho \u)=0,\\
&\rho(\partial_t\mathbf{u} + \u\cdot \nabla \u) -  \mu \Delta \u -(\lambda+\mu) \nabla \div \u+\nabla P +\frac12\nabla m^2=  0,\\
&\partial_tm+\div(m \u)=0 .
\end{aligned}\right.
\end{eqnarray}
Clearly $(\rho^{(0)}, \u^{(0)}, m^{(0)})$ with
$$
\rho^{(0)}=1, \quad \u^{(0)}={0}, \quad  m^{(0)}=1
$$
solves (\ref{0mm2}). We intend to understand the well-posedness and stability problem on the system
governing the perturbation $(a, \u, b)$, where
$$
a =\rho -1, \quad b=m-1.
$$
It is easy to check that $(a, \u, b)$ satisfies
\begin{eqnarray}\label{333mm3}
\left\{\begin{aligned}
&\partial_ta+\div \u+\u\cdot\nabla a+a\,\div \u=0,\\
&\partial_t\u + \u\cdot \nabla \u -  \mu\Delta \u - (\lambda+\mu)\nabla \div \u+\nabla P(1+a) +\frac12\nabla {(b+1)^2}=\mathbf{M}(a,\u,{b}),\\
&\partial_t{b}+\div \u+\u\cdot\nabla {b}+{b}\,\div \u=0,\\
&(a,\u,b)|_{t=0}=(a_0,\u_0,b_0),
\end{aligned}\right.
\end{eqnarray}
with
\begin{align}\label{}
\mathbf{M}(a,\u,{b})\stackrel{\mathrm{def}}{=}&\frac{a}{1+a}(\nabla P(1+a) +\frac12\nabla {(b+1)^2}-(\mu\Delta \u+(\lambda+\mu)\nabla\div \u)).
\end{align}
As the first step of our main results, we provide a local well-posedness result in the Besov space.
\begin{proposition}(Local well-posedness)\quad \label{dingli1}
Let $1<p<4$. Assume  $\u_0\in \dot{B}_{p,1}^{\frac 2p-1}(\R^2)$,  $(a_0,b_0)\in \dot{B}_{p,1}^{\frac 2p}(\R^2)$ with $1 + a_0$ bounded away from zero. Then there exists a positive time
$T$ such that the system \eqref{333mm3} has a unique solution $(a, \u, b)$ satisfying
\begin{align*}
&(a,b)\in C([0,T ];{\dot{B}}_{p,1}^{\frac {2}{p}}),\quad \u\in C([0,T ];{\dot{B}}_{p,1}^{\frac {2}{p}-1})\cap L^{1}
([0,T];{\dot{B}}_{p,1}^{\frac 2p+1}).
\end{align*}
\end{proposition}

Before stating our main results, we introduce some notation.
Let $\mathcal{S}(\R^2)$ be the Schwartz space on $\R^2$ and $\mathcal{S}'(\R^2)$ be its dual
space. For any $z \in\mathcal{S}'(\R^2)$,
the lower and higher frequency parts are expressed as\footnote{Note that for technical reasons, we need a small
overlap between low and high frequency.}
\begin{align*}
z^\ell\stackrel{\mathrm{def}}{=}\sum_{j\leq j_0}\ddj z\quad\hbox{and}\quad
z^h\stackrel{\mathrm{def}}{=}\sum_{j\ge j_0-1}\ddj z
\end{align*}
for some fixed   integer $j_0$ (the value of $j_0$ is fixed in  the proofs of the main theorems).
The corresponding truncated semi-norms are defined  as follows:
\begin{align*}
\norm{z}^\ell_{\dot B^{s}_{p,1}}
\stackrel{\mathrm{def}}{=}  \norm{z^{\ell}}_{\dot B^{s}_{p,1}}
\ \hbox{ and }\   \norm{z}^{h}_{\dot B^{s}_{p,1}}
\stackrel{\mathrm{def}}{=}  \norm{z^{h}}_{\dot B^{s}_{p,1}}.
\end{align*}
Let $\mathbb P= I -\nabla\Delta^{-1} \nabla\cdot$ be the projection onto the divergence-free vector
fields and $\mathbb Q= I -\mathbb P=  \nabla\Delta^{-1} \nabla\cdot$.

The small data global well-posedness and stability result on (\ref{333mm3}) is stated in the
following theorem.
\begin{theorem}(Global  well-posedness)\quad \label{dingli2}
Let   $2\le p<4.$
 For any  $(a_0^\ell,\q\u_0^\ell,b_0^\ell)\in \dot{B}_{2,1}^{0}(\R^2)$,    $(a^h_0, b_0^h)\in \dot{B}_{p,1}^{\frac 2p}(\R^2)$ and $(\p\u_0,\q\u_0^h)\in \dot{B}_{p,1}^{\frac 2p-1}(\R^2)$,
 there
exists a positive constant $c_0$ such that,  if
\begin{align}\label{smallness}
\norm{(a^{\ell}_0,\q\u^{\ell}_0,{b}^{\ell}_0)}_{\dot{B}^{0}_{2,1}}
+\norm{(a^h_0,{b}^h_0)}_{\dot{B}^{\frac {2}{p}}_{p,1}}+\norm{(\p\u_0,\q\u^h_0)}_{\dot{B}^{\frac {2}{p}-1}_{p,1}}\leq c_0,
\end{align}
then
the system \eqref{333mm3} has a unique global solution $(a,\u,b)$ so that
\begin{align*}
&(a^\ell,b^\ell)\in C_b(\R^+;{\dot{B}}_{2,1}^{0}),\quad
 (a^h,b^h)\in C_b(\R^+;{\dot{B}}_{p,1}^{\frac {2}{p}}),\\
&\q\u^\ell\in C_b(\R^+;{\dot{B}}_{2,1}^{0}\cap L^{1}
(\R^+;{\dot{B}}_{2,1}^{2}),\quad (\p\u,\q\u^h)\in C_b(\R^+;{\dot{B}}_{p,1}^{\frac {2}{p}-1}\cap L^{1}
(\R^+;{\dot{B}}_{p,1}^{\frac 2p+1}).
\end{align*}
Moreover,  there exists some constant $C$ such that
\begin{align}\label{xiaonorm}
&\mathcal{X}(t)\leq Cc_0,
\end{align}
where
\begin{align*}
\mathcal{X}(t)\stackrel{\mathrm{def}}{=}&
\norm{(a^{\ell},\q\u^{\ell},{b}^{\ell})}_{\widetilde{L}^{\infty}_t(\dot{B}^{0}_{2,1})}
+\norm{(a^h,{b}^h)}_{\widetilde{L}^{\infty}_t(\dot{B}^{\frac {2}{p}}_{p,1})}
+\norm{(\p\u,\q\u^h)}_{\widetilde{L}^{\infty}_t(\dot{B}^{\frac {2}{p}-1}_{p,1})}
\notag\\
&+\norm{(\ta^{\ell},\q\u^{\ell})}_{L^1_t(\dot{B}^{2}_{2,1})}
+\norm{\varphi^h}_{L^{1}_t(\dot{B}^{\frac {2}{p}}_{p,1})}
+\norm{(\p\u,\q\u^h)}_{L^1_t(\dot{B}^{\frac {2}{p}+1}_{p,1})}.
\end{align*}
One refers to \eqref{tade} for the definition of $\varphi$.
\end{theorem}
\subsection{Strategy of the proof of Theorem  \ref{dingli2}}
Let us point out new ingredients in the proof of Theorem \ref{dingli2}. For usual compressible
Navier-Stokes equations (see for example \cite{danchin2000, helingbing}), the major difficulty stems from the
convection term in the density equation, as it may cause a loss of one derivative of the density. To overcome it, previous proofs heavily relied on a paralinearized version combined with a Lagrangian change of variables. For the compressible viscous non-resistive MHD system \eqref{333mm3}, the situation becomes more complicated.
There are absence of dissipation in  the density  equation and the magnetic field equation, we cannot get  any smoothing effect of the  density and the magnetic field. This bring us big difficulty to construct the global solutions of the system.
 The new ingredient in the present paper lies in the introduction of un unknown good function $\varphi$ (see \eqref{tade}),
which enables us to capture the dissipation arising from combination of density and the magnetic field.
Finally, we complete the proof of Theorem \ref{dingli2} by a continuous argument.

It is natural and physically important to study the large-time behavior of the global solution
obtained in \eqref{dingli2}. The large-time behavior has always been a prominent topic
on the fluid equations. Important
results have been established for the compressible Navier-Stokes equations (see, e.g.,
\cite{xujiang, xujiang2021jde, zhaixiaoping}) and the compressible MHD equations (see, e.g.,
\cite{huxianpeng, lifucai}).

\vskip .1in
What is special here is that the system concerned here is partially dissipated with
no damping or dissipation in the equations of $\rho$ and $b$.
We show that, when the low modes of the initial data are in a Besov space with suitable
negative index, then the Sobolev norm of the solution is shown to decay at an optimal rate.
The proof relies on the enhanced dissipation resulting from the interaction between the velocity
and the magnetic field.

\begin{theorem}$\mathrm{(}$Optimal decay$\mathrm{)}$\quad\label{dingli3}
Let $\Lambda^sz\stackrel{\mathrm{def}}{=} \mathcal{F}^{-1}(|\xi|^s\mathcal{F}z)(s\in\R)$
 and
 ~$(a,\u,b)$ be the global small solutions addressed by Theorem \ref{dingli2} with $p=2$. For any $0<\sigma\le1,$
if additionally the initial data satisfying $(a_0^\ell,\u_0^\ell,b_0^\ell)\in{\dot{B}_{2,\infty}^{-\sigma}}(\R^2)$, then we have the following time-decay rate
\begin{align}\label{huichen}
\norm{\Lambda^{\gamma_1} (\varphi,\u)}_{L^2}
\le C(1+t)^{-\frac{\gamma_1+\sigma}{2}},\quad\quad &\forall \gamma_1\in\left(-\sigma,0\right]
\end{align}
with $\varphi$ is defined in \eqref{tade}.
\end{theorem}

\begin{remark}\label{yaya3}
The above decay rate \eqref{huichen} coincides with the  heat flows, thus it is optimal in some sense.
\end{remark}

Finally, we  mention the small data global well-posedness result for a closely related
system of inhomogeneous incompressible MHD equations. The general inhomogeneous incompressible MHD equations are of the form
\begin{eqnarray}\label{nmm1}
\left\{\begin{aligned}
&\partial_t\rho+\div(\rho \v)=0,\\
&\rho(\partial_t\mathbf{v} + \v\cdot \nabla \v) -  \mu \Delta \v +\nabla P=  (\nabla\times \mathbf{B})\times \mathbf{B},\\
&\partial_t\mathbf{B}-\nabla\times(\v\times\mathbf{B})=0,\\
&\div\v=\div\mathbf{B}=0,\\
&(\rho,\v,\mathbf{B})|_{t=0}=(\rho_0,\v_0,\mathbf{B}_0).
\end{aligned}\right.
\end{eqnarray}
If we set
\begin{align*}
&\v=(\v^1(t,x_1,x_2),\v^2(t,x_1,x_2),0)\stackrel{\mathrm{def}}{=}(\u,0),\nn\\
&\rho\stackrel{\mathrm{def}}{=}\rho(t,x_1,x_2),\quad\mathbf{B}\stackrel{\mathrm{def}}{=}(0,0,b(t,x_1,x_2)),
\end{align*}
then \eqref{nmm1} is reduced to
\begin{eqnarray}\label{nmm2}
\left\{\begin{aligned}
&\partial_t\rho+\div(\rho \u)=0,\\
&\rho(\partial_t\mathbf{u} + \u\cdot \nabla \u) -  \mu \Delta \u +\nabla P +\frac12\nabla b^2=  0,\\
&\partial_tb+\div(b \u)=0 ,\\
&\div\u=0,\\
&(\rho,\u,b)|_{t=0}=(\rho_0,\u_0,b_0).
\end{aligned}\right.
\end{eqnarray}

Different from the compressible MHD equations, the combination  $\Pi:= P+\frac12 b^2$ can be regarded as new pressure and the new system \eqref{nmm2} is decoupled into equations of $(\rho,\u)$ and the equation of $b$. We can solve the equations of $(\rho,\u)$ first and then get the solution of $b$ through the third equation of \eqref{nmm2}.
Now we write $\rho=1+a$, inspired by \cite{abidi}, \cite{xuhuan} and the previous well-posedness result on the compressible MHD equations, we obtain the following global well-posedness result on \eqref{nmm2}.
We shall not provide a detailed proof for this result.

\begin{theorem}\label{dingli4}
Let $p\in(1,4)$, $(a_0,\u_0,b_0)\in\dot{B}_{p,1}^{\frac{2}{p}}(\R^2)$
with $\div \u_0=0$ and  $1 + a_0$ bounded away from zero.  Then \eqref{nmm2} has a unique global solution $(a,\u,\nabla\Pi,b)$ such that for any $t>0$,
\begin{align*}
(a,b)\in& C(\R^+;\dot{B}_{p,1}^{\frac{2}{p}}(\R^2))\cap \widetilde{L}_t^\infty(\dot{B}_{p,1}^{\frac{2}{p}}(\R^2)),
\ \nabla\Pi\in L_t^1(\dot{B}_{p,1}^{\frac{2}{p}-1}(\R^2)),\nonumber\\
\u\in& C(\R^+;;\dot{B}_{p,1}^{\frac{2}{p}-1}(\R^2))\cap \widetilde{L}_t^\infty(\dot{B}_{p,1}^{\frac{2}{p}-1}(\R^2))
\cap L_t^1(\dot{B}_{p,1}^{\frac{2}{p}+1}(\R^2)).
\end{align*}
Moreover, we have
\begin{align*}
\|(a,b)\|_{\widetilde{L}_t^\infty(B_{p,1}^{\frac{2}{p}})}+\|\u\|_{\widetilde{L}_t^\infty(\dot{B}_{p,1}^{\frac{2}{p}-1})}+\|\u\|_{L_t^1(\dot{B}_{p,1}^{\frac{2}{p}+1})}
+\|\nabla\Pi\|_{L_t^1(\dot{B}_{p,1}^{\frac{2}{p}-1})}\leq C\exp\left(C\exp\big(Ct^\frac12\big)\right)
\end{align*}
for some time-independent constant $C$.
\end{theorem}

The rest of this paper is arranged as follows. In the second section, we  recall some basic facts about Littlewood-Paley theory.
In the third section, we use the fixed point theorem  to outline the proof of Proposition \ref{dingli1}.
In the forth section, we use three subsections to prove Theorem \ref{dingli2}.
 In the first subsection,  we  exploit the special  structure of \eqref{333mm3} to capture the dissipation arising from combination of density and the magnetic field at low frequencies part and
 in the second subsection, we introduce
   a so called effective velocity to capture  the dissipation arising from combination of density and the magnetic field at high frequencies part, respectively.
 In the last subsection, we use the continuity argument to close  the energy estimates
and thus complete the proof of Theorem \ref{dingli2}.
We shall prove the Theorem \ref{dingli3} in Section 5. Inspired by the papers \cite{xujiang2021jde},
 our main task is to establish a Lyapunov-type inequality in time for energy norms
(see \eqref{sa58}) by using the pure energy argument (independent of spectral analysis).

\medskip
Let us introduce some  notations.
For two operators $A$ and $B$, we denote $[A,B]=AB-BA$, the commutator between $A$ and $B$.
The letter $C$ stands for a generic constant whose meaning is clear from the context.
We denote $\langle a,b\rangle$ the $L^2(\R^2)$ inner product of $a$ and $b$
 and write $a\lesssim b$ instead of $a\leq Cb$.
Given a Banach space $X$, we shall denote $\norm{(a,b)}_{X}=\norm{a}_{X}+\norm{b}_{X}$.

\bigskip
\section{ Preliminaries }

This section reviews Besov spaces and related facts to be used in the subsequent sections.
We start with the Littlewood-Paley decomposition.
To define it,  we fix a smooth radial non-increasing function $\chi$
supported in the ball $B(0,\frac 43)$ of $\R^2,$ and with value $1$ on $B(0,\frac34)$ such that, for
$\varphi(\xi)=\chi(\frac{\xi}{2})-\chi(\xi),$
$$
\qquad\sum_{j\in\Z}\varphi(2^{-j}\cdot)=1\ \hbox{ in }\ \R^2\setminus\{0\}
\quad\hbox{and}\quad \mathrm{Supp}\,\varphi\subset \Big\{\xi\in\R^2 : \frac34\leq|\xi|\leq\frac83\Big\}\cdotp
$$
The homogeneous dyadic blocks $\dot{\Delta}_j$ are defined on tempered distributions by
$$\dot{\Delta}_j u\stackrel{\mathrm{def}}{=}\varphi(2^{-j}D)u\stackrel{\mathrm{def}}{=}{\mathcal F}^{-1}(\varphi(2^{-j}\cdot){\mathcal F} u).
$$
For any homogeneous function $A$ of order 0 and smooth outside 0, we have
\begin{equation*}\label{}
\forall p\in[1,\infty],\quad\quad\|\ddj (A(D) u)\|_{L^p}\le C\|\ddj u\|_{L^p}.
\end{equation*}
\begin{definition}
Let $p,r$ be in~$[1,+\infty]$ and~$s$ in~$\R$, $u\in\mathcal{S}'(\R^2)$. We define the Besov norm by
$$
\|u\|_{{\dot{B}^s_{p,r}}}\stackrel{\mathrm{def}}{=}\big\|\big(2^{js}\|\ddj
u\|_{L^{p}}\big)_j\bigr\|_{\ell ^{r}({\mathop{\mathbb Z\kern 0pt}\nolimits})}.
$$
We then define the homogeneous Besov spaces by
$\dot{B}_{p,r}^s\stackrel{\mathrm{def}}{=}\left\{u\in\mathcal{S}'_h(\R^2),
\|u\|_{\dot{B}_{p,r}^s}<\infty\right\}$, where $u\in \mathcal{S}'_h(\R^2)$ means that $u\in \mathcal{S}'(\R^2)$ and $\lim_{j\to -\infty}\|\dot{S}_ju\|_{L^\infty}=0$ (see Definition 1.26 of \cite{bcd}).
\end{definition}

When employing parabolic estimates in Besov spaces, it is somehow natural to take the time-Lebesgue norm before performing the summation for computing the Besov norm. So we next introduce the following Besov-Chemin-Lerner space $\widetilde{L}_T^q(\dot{B}_{p,r}^s)$ (see\,\cite{bcd}):
$$
\widetilde{L}_T^q(\dot{B}_{p,r}^s)={\Big\{}u(t,x)\in (0,+\infty)\times\mathcal{S}'_h(\R^2):
\|u\|_{\widetilde{L}_T^q(\dot{B}_{p,r}^s)}<+\infty{\Big\}},
$$
where
$$
\|u\|_{\widetilde{L}_T^q(\dot{B}_{p,r}^s)}\stackrel{\mathrm{def}}{=}\bigl{\|}2^{ks}\|\dot{\Delta}_k u(t)\|_{L^q(0,T;L^p)}\bigr{\|}_{\ell^r}.
$$
The index $T$ will be omitted if $T=+\infty$ and we shall denote by $\widetilde{\mathcal{C}}_b([0,T]; \dot{B}^s_{p,r})$ the subset of functions of $\widetilde{L}^\infty_T(\dot{B}^s_{p,r})$ which are also continuous from
$[0,T]$ to $\dot{B}^s_{p,r}$.

By the Minkowski inequality, we have the following inclusions between the
Chemin-Lerner space ${\widetilde{L}^\lambda_{T}(\dot{B}_{p,r}^s)}$ and the Bochner space ${{L}^\lambda_{T}(\dot{B}_{p,r}^s)}$:
\begin{align*}
\|u\|_{\widetilde{L}^\lambda_{T}(\dot{B}_{p,r}^s)}\le\|u\|_{L^\lambda_{T}(\dot{B}_{p,r}^s)}\hspace{0.5cm} \mathrm{if }\hspace{0.2cm}  \lambda\le r,\hspace{0.5cm}
\|u\|_{\widetilde{L}^\lambda_{T}(\dot{B}_{p,r}^s)}\ge\|u\|_{L^\lambda_{T}(\dot{B}_{p,r}^s)},\hspace{0.5cm} \mathrm{if }\hspace{0.2cm}  \lambda\ge r.
\end{align*}
The following Bernstein's lemma will be repeatedly used throughout this paper.

\begin{lemma}\label{bernstein}
Let $\mathcal{B}$ be a ball and $\mathcal{C}$ a ring of $\mathbb{R}^2$. A constant $C$ exists so that for any positive real number $\lambda$, any
non-negative integer k, any smooth homogeneous function $\sigma$ of degree m, and any couple of real numbers $(p, q)$ with
$1\le p \le q\le\infty$, there hold
\begin{align*}
&&\mathrm{Supp} \,\hat{u}\subset\lambda \mathcal{B}\Rightarrow\sup_{|\alpha|=k}\|\partial^{\alpha}u\|_{L^q}\le C^{k+1}\lambda^{k+2(\frac1p-\frac1q)}\|u\|_{L^p},\\
&&\mathrm{Supp} \,\hat{u}\subset\lambda \mathcal{C}\Rightarrow C^{-k-1}\lambda^k\|u\|_{L^p}\le\sup_{|\alpha|=k}\|\partial^{\alpha}u\|_{L^p}
\le C^{k+1}\lambda^{k}\|u\|_{L^p},\\
&&\mathrm{Supp} \,\hat{u}\subset\lambda \mathcal{C}\Rightarrow \|\sigma(D)u\|_{L^q}\le C_{\sigma,m}\lambda^{m+2(\frac1p-\frac1q)}\|u\|_{L^p}.
\end{align*}
\end{lemma}

Next we  recall a few nonlinear estimates in Besov spaces which may be
obtained by means of paradifferential calculus.
Here, we recall the decomposition in the homogeneous context:
\begin{align}\label{bony}
uv=\dot{T}_uv+\dot{T}_vu+\dot{R}(u,v)=\dot{T}_uv+\dot{T}'_vu,
\end{align}
where
$$\dot{T}_uv\stackrel{\mathrm{def}}{=}\sum_{j\in \mathbb{Z}}\dot{S}_{j-1}u\dot{\Delta}_jv, \hspace{0.5cm}\dot{R}(u,v)\stackrel{\mathrm{def}}{=}\sum_{j\in \Z}
\dot{\Delta}_ju\widetilde{\dot{\Delta}}_jv,$$
and
$$ \widetilde{\dot{\Delta}}_jv\stackrel{\mathrm{def}}{=}\sum_{|j-j'|\le1}\dot{\Delta}_{j'}v,\hspace{0.5cm} \dot{T}'_vu\stackrel{\mathrm{def}}{=}\sum_{j\in \Z}\dot{S}_{j+2}v\dot{\Delta}_ju.$$

The paraproduct $\dot{T}$ and the remainder $\dot{R}$ operators satisfy the following
continuous properties.

\begin{lemma}\label{fangji}
Let $(s,r)\in\R\times[1,\infty]$  and $1\leq p,p_1,p_2\leq\infty$ with  $\frac 1 p=\frac{1}{p_1}+\frac{1}{p_2}.$
\begin{itemize}
\item We have:
$$
\|\dot{T}_uv\|_{\dot B^s_{p,r}}\lesssim \|u\|_{L^{p_1}}\|v\|_{\dot B^{s}_{p_2,r}}\quad\hbox{and}\quad
\|\dot{T}_uv\|_{\dot B^{s+t}_{p,r}}\lesssim \|u\|_{\dot B^t_{p_1,\infty}}\|v\|_{\dot B^{s}_{p_2,r}},\quad
\hbox{if }\ t<0.
$$
\item If
$s_1+s_2>0$ and $\frac{1}{r}=\frac{1}{r_1}+\frac{1}{r_2}\leq1$ then
 $$\|\dot{R}(u,v)\|_{\dot B^{s_1+s_2}_{p,r}}\lesssim\|u\|_{\dot B^{s_1}_{p_1,r_1}}
\|v\|_{\dot B^{s_2}_{p_2,r_2}}.$$
\item If $s_1+s_2=0$  and $\frac{1}{r_1}+\frac{1}{r_2}\geq1$ then
\begin{align}\label{reminder}
\|\dot{R}(u,v)\|_{\dot B^{0}_{p,\infty}}\lesssim\|u\|_{\dot B^{s_1}_{p_1,r_1}}
\|v\|_{\dot B^{s_2}_{p_2,r_2}}.
\end{align}
\end{itemize}

 \end{lemma}
From Lemma \ref{fangji}, we may deduce the following  several nonlinear estimates in Besov spaces
\begin{lemma}\label{guangyidaishu}{\rm(\cite{bcd})}
 Let $s_1\le \frac{2}{p}, s_2< \frac{2}{p},\, s_1+s_2\ge 2\max(0,\frac 2p-1)$,
and $1\le p\le \infty$.
Assume that  $u\in\dot B^{s_1}_{p,1}(\R^2)$ and
$v\in \dot B^{s_2}_{p,\infty}(\R^2)$. Then there
holds
\begin{align*}
\|uv\|_{\dot B^{s_1+s_2-\frac{2}{p}}_{p,\infty}}
\le C\|u\|_{\dot B^{s_1}_{p,1}}\|v\|_{\dot B^{s_2}_{p,\infty}}.
\end{align*}
\end{lemma}

\begin{lemma}\label{law}{\rm(\cite[Proposition A.1]{xujiang2021jde})}
Let $1\leq p, q\leq \infty$, $s_1\leq \frac{2}{q}$, $s_2\leq 2\min\{\frac1p,\frac1q\}$ and $s_1+s_2>2\max\{0,\frac1p +\frac1q -1\}$. For any $(u,v)\in\dot{B}_{q,1}^{s_1}({\mathbb R} ^2)\times\dot{B}_{p,1}^{s_2}({\mathbb R} ^2)$, we have
\begin{align*}
\|uv\|_{\dot{B}_{p,1}^{s_1+s_2 -\frac{2}{q}}}\lesssim \|u\|_{\dot{B}_{q,1}^{s_1}}\|v\|_{\dot{B}_{p,1}^{s_2}}.
\end{align*}
\end{lemma}

\begin{lemma}\label{daishu}
Let
$2\leq p <4$. For any
 $u\in\dot{B}_{p,1}^{\frac {2}{p}}(\R^2), v^\ell\in\dot{B}_{2,1}^{0}(\R^2)$ and $ v^h\in\dot{B}_{p,1}^{\frac {2}{p}-1}(\R^2),$ we have
\begin{align}\label{pengyou}
&\|(uv)^\ell\|_{\dot{B}_{2,1}^{0}}\lesssim(\|v^\ell\|_{\dot{B}_{2,1}^{0}}+\|v^h\|_{\dot{B}_{p,1}^{\frac {2}{p}-1}})\|u\|_{\dot{B}_{p,1}^{\frac {2}{p}}}.
\end{align}
\end{lemma}

\begin{proof}
We first use  Bony's decomposition to write
\begin{align}\label{D9}
\dot S_{j_0+1} (uv)=\dot{T}_{u}\dot S_{j_0+1}  v+\dot S_{j_0+1} \bigl(\dot{T}_{v} u+ \dot{R}(v,u)\bigr)
+[\dot S_{j_0+1} ,\dot{T}_{u}]v.
\end{align}
Applying Lemma \ref{fangji}, we have
\begin{align}\label{D11-1}
\|\dot{T}_{u}\dot S_{j_0+1}  v\|_{\dot{B}_{2,1}^{0}}\lesssim&\|u\|_{L^\infty}\|v^\ell\|_{\dot{B}_{2,1}^{0}}
\lesssim\|v^\ell\|_{\dot{B}_{2,1}^{0}}\|u\|_{\dot{B}_{p,1}^{\frac {2}{p}}},
\end{align}
and, for  $\frac  1{p*}=\frac  12-\frac1p$,
\begin{align}\label{D116}
\|\dot S_{j_0+1} \dot{T}_{v} u\|_{\dot{B}_{2,1}^{0}}
\lesssim\|v\|_{\dot{B}_{p^*,1}^{\frac {2}{p^*}-1}}\|u\|_{\dot{B}_{p,1}^{\frac {2}{p}}}
\lesssim\|v\|_{\dot{B}_{p,1}^{\frac {2}{p}-1}}\|u\|_{\dot{B}_{p,1}^{\frac{2}{p}}}.
\end{align}
For the reminder term $\|\dot S_{j_0+1}  \dot{R}(v,u)\|_{\dot{B}_{2,1}^{0}}$, we cannot use Lemma \ref{fangji} directly, however, in view of the fact that $1\leq \frac{p}{2} <2$ and $ \frac{4}{p}- 1 >0 $, there holds
\begin{align}\label{D222}
\|\dot S_{j_0+1}  \dot{R}(v,u)\|_{\dot{B}_{2,1}^{0}}
\lesssim\|\dot{R}(v,u)\|_{\dot{B}_{{p}/{2},1}^{\frac {4}{p}-1}}
\end{align}
from which and Lemma \ref{fangji}, we get
\begin{align}\label{D323}
\|\dot S_{j_0+1}  \dot{R}(v,u)\|_{\dot{B}_{2,1}^{0}}
\lesssim\|v\|_{\dot{B}_{p,1}^{\frac {2}{p}-1}}\|u\|_{\dot{B}_{p,1}^{\frac{2}{p}}}.
\end{align}
By Lemma 6.1 in \cite{helingbing},  the term with the commutator can be bounded
\begin{align}\label{D11}
\|[\dot S_{j_0+1} ,\dot{T}_{u}]v\|_{\dot{B}_{2,1}^{0}}\lesssim&\|\nabla u\|_{\dot{B}_{p^*,1}^{\frac {2}{p^*}-1}}\|v\|_{\dot{B}_{p,1}^{\frac {2}{p}-1}}
\lesssim\|v\|_{\dot{B}_{p,1}^{\frac {2}{p}-1}}\|u\|_{\dot{B}_{p,1}^{\frac {2}{p}}}.
\end{align}
Thus, the combination of (\ref{D9})--(\ref{D11}) shows the validity of  \eqref{pengyou}.
\end{proof}

We also need the following classical commutator's estimate.
\begin{lemma}{\rm(\cite[Lemma 2.100]{bcd})}\label{jiaohuanzi}
Let $1\leq p\leq \infty$, $-2\min\left\{\frac1p,1-\frac{1}{p}\right\}<s\leq \frac 2p$. For any
$v\in \dot{B}_{p,1}^{s}(\R^2)$ and $\nabla u\in \dot{B}_{p,1}^{\frac {2}{p}}(\R^2)$,
 there holds
$$
\big\|[\dot{\Delta}_j,u\cdot \nabla ]v\big\|_{L^p}\lesssim d_j 2^{-js}\|\nabla u\|_{\dot{B}_{p,1}^{\frac {2}{p}}}\|v\|_{\dot{B}_{p,1}^{s}}
$$
where $(d_j)_{\ell^1}=1$.
\end{lemma}

Finally, we recall a composition result and the parabolic regularity estimate for the
heat equation to end this section.
  \begin{lemma}\label{fuhe}{\rm(\cite{bcd})}
   Let $G$ with $G(0)=0$ be a smooth function defined on an open interval $I$
of $\R$ containing~$0.$
Then  the following estimates
$$
\|G(f)\|_{\dot B^{s}_{p,1}}\lesssim\|f\|_{\dot B^s_{p,1}}\quad\hbox{and}\quad
\|G(f)\|_{\widetilde{L}^q_T(\dot B^{s}_{p,1})}\lesssim\|f\|_{\widetilde{L}^q_T(\dot B^s_{p,1})}
$$
hold true for  $s>0,$ $1\leq p,\, q\leq \infty$ and
 $f$  valued in a bounded interval $J\subset I.$
\end{lemma}
\begin{lemma} [\cite{bcd}]\label{heat}
Let $\sigma\in \R$,  $T>0$, $1\leq p,r\leq\infty$ and $1\leq q_{2}\leq q_{1}\leq\infty$.  Let $u$  satisfy the heat equation
$$\partial_tu-\Delta u=f,\quad
u|_{t=0}=u_0.$$
Then  there holds the following a priori estimate
\begin{align*}
\|u\|_{\widetilde L_{T}^{q_1}(\dot B^{\sigma+\frac 2{q_1}}_{p,r})}\lesssim
\|u_0\|_{\dot B^\sigma_{p,r}}+\|f\|_{\widetilde L^{q_2}_{T}(\dot B^{\sigma-2+\frac 2{q_2}}_{p,r})}.
\end{align*}
\end{lemma}

\bigskip

\section{The proof of Proposition \ref{dingli1}}

We prove Proposition \ref{dingli1} by a fixed point theorem under the  Lagrangian coordinates. We follow the paper \cite{danchin2014} closely and only give the sketch of the proof.

{\bf Step 1.} First, we convert \eqref{0mm2} into is Lagrangian formulation. For this, we need to introduce some notations. For a vector $\w=\w(x)=(w_1,w_2)$, $\nabla_x\w$ denotes the matrix $(\partial_{x_i}w_j)_{ij}$ and $D_x\w=(\nabla_x\w)^{\intercal}$ (i.e., the transpose of $\nabla_x\w$). We may also frequently write $\nabla\w$ and $D\w$ when it is clear which space variable $\w$ depends on.

If $\u=\u(t,x)$ is a $C^1$ vector field, it uniquely determines a trajectory $X(t,\cdot)$, defined by the ODE
\begin{eqnarray}\label{define trajectory via Eulerian velocity}
\left\{\begin{aligned}
&\frac{d}{dt}X(t,y)=\u(t,X(t,y)),\\
&X(0,y)=y.
\end{aligned}\right.
\end{eqnarray}
Moreover, $X(t,\cdot)$ is a $C^1$-diffeomorphism over $\R^2$ for every $t\ge0$.

We are about to reformulate \eqref{0mm2} using the following unknowns in Lagrangian coordinates:
\begin{align}\label{from Eulerian to Lagrangian}
\bar{\rho}(t,y)\stackrel{\mathrm{def}}{=}\rho\big(t,X(t,y)\big),\quad\bar{m}(t,y)\stackrel{\mathrm{def}}{=}m\big(t,X(t,y)\big), \quad
\hbox{and}\quad\bar{\u}(t,y)\stackrel{\mathrm{def}}{=}\u\big(t,X(t,y)\big).
\end{align}
We may keep in mind that now $X$ only depends on $\bar{\u}$ since
\begin{align}\label{trj2}
X(t,y)=y+\int_0^t\bar{\u}(\tau,y)\,d\tau.
\end{align}

Next, let us introduce $J_{\bar{\u}}(t,y)=\det DX(t,y)$. Then $\eqref{0mm2}_1$ and $\eqref{0mm2}_3$ imply, respectively,
\begin{align}\label{2const}
J_{\bar{\u}}\bar{\rho}\equiv\rho_0,\ \ {\rm and}\ \ J_{\bar{\u}}\bar{m}\equiv m_0.
\end{align}
To reformulate $\eqref{0mm2}_2$, we further introduce
$A_{\bar{\u}}(t,y)=\big(DX(t,y)\big)^{-1}$ and $\mathscr{A}_{\bar{\u}}(t,y)={\mathop{\mbox{\rm adj}}} DX(t,y)$ (the adjugate of $D X$, i.e.,  $\mathscr{A}_{\bar{\u}}=J_{\bar{\u}}A_{\bar{\u}}$).
As in \cite{danchin2014}, evaluating $\eqref{0mm2}_2$ at $(t,X(t,y))$, multiplying the resulting equation by $J_{\bar{\u}}$, and using \eqref{2const}, we have
\begin{eqnarray}\label{Lagrangian formulation of compressible pressureless flow}
\left\{\begin{aligned}
&\rho_0\partial_t \bar{\u}-\mu\div(\mathscr{A}_{\bar{\u}}A_{\bar{\u}}^{\intercal}\nabla\bar{\u})
-(\mu+\lambda)\mathscr{A}_{\bar{\u}}^{\intercal}\nabla\mathrm{Tr}(A_{\bar{\u}}D\bar{\u})+\mathscr{A}_{\bar{\u}}^{\intercal}\nabla (P(J_{\bar{\u}}^{-1}\rho_0)+\frac12(J_{\bar{\u}}^{-1}m_0)^2)=0,\\
&\bar{\u}|_{t=0}=\u_0.
\end{aligned}\right.
\end{eqnarray}
Here, $\mathrm{Tr}$ denotes the trace of square matrices.

{\bf Step 2.} Linearized system. Note that \eqref{Lagrangian formulation of compressible pressureless flow} is already a determined system with $\bar{\u}$ the only unknown. Since it is fully nonlinear, we need to reformulate it as
\begin{eqnarray}\label{linearized}
\left\{\begin{aligned}
&\rho_0\partial_t \bar{\u}-\mu\Delta\bar{\u}-(\mu+\lambda)\nabla\div\bar{\u}=f(\bar{\u}),\\
&\bar{\u}|_{t=0}=\u_0,
\end{aligned}\right.
\end{eqnarray}
where
\begin{align*}
f(\bar{\u})=&\mu\div((\mathscr{A}_{\bar{\u}}A_{\bar{\u}}^{\intercal}-I)\nabla\bar{\u})
+(\mu+\lambda)(\mathscr{A}_{\bar{\u}}^{\intercal}-I)\nabla\mathrm{Tr}(A_{\bar{\u}}D\bar{\u})\\
&+(\mu+\lambda)\nabla\mathrm{Tr}((A_{\bar{\u}}-I)D\bar{\u})-\mathscr{A}_{\bar{\u}}^{\intercal}\nabla (P(J_{\bar{\u}}^{-1}\rho_0)+\frac12(J_{\bar{\u}}^{-1}m_0)^2)
\end{align*}
and $I$ is the identity matrix.

We need the following well-posedness result for the linearized system.
\begin{theorem}[See \cite{xhjee}]\label{wellposedness for linear system}
Let $1<p<4$. Assume  $\u_0\in \dot{B}_{p,1}^{\frac 2p-1}(\R^2)$,  $\rho_0-1\in \dot{B}_{p,1}^{\frac 2p}(\R^2)$, and $\inf\rho_0>0$. If $f\in L^{1}
([0,T];{\dot{B}}_{p,1}^{\frac2p-1})$ for some positive time $T$, then the system
\begin{eqnarray*}
\left\{\begin{aligned}
&\rho_0\partial_t \bar{\u}-\mu\Delta\bar{\u}-(\mu+\lambda)\nabla\div\bar{\u}=f,\\
&\bar{\u}|_{t=0}=\u_0
\end{aligned}\right.
\end{eqnarray*}
has a unique solution $\bar{\u}$ in the class $C([0,T ];{\dot{B}}_{p,1}^{\frac {2}{p}-1})\cap L^{1}
([0,T];{\dot{B}}_{p,1}^{\frac 2p+1})$.

Moreover, we have the global estimate
\begin{align*}
\|\bar{\u}\|_{L_T^\infty(\dot{B}_{p,1}^{\frac{2}{p}-1})}+\|\partial_t\bar{\u},\Delta\bar{\u}\|_{L_T^1(\dot{B}_{p,1}^{\frac{2}{p}-1})}\le C\|\u_0\|_{\dot{B}_{p,1}^{\frac{2}{p}-1}}+C\|f\|_{L_T^1(\dot{B}_{p,1}^{\frac{2}{p}-1})},
\end{align*}
where $C$ depends on $p,\inf\rho_0,\mu,\lambda,\|\rho_0-1\|_{\dot{B}_{p,1}^{\frac2p}}$ but $T$.
\end{theorem}

In fact, the author in \cite{xhjee} did not discuss the case $p=2$. However, our regularity of the initial density $\rho_0$ is much higher than that in \cite{xhjee}. Then one can follow the argument in \cite{xhjee} to show Theorem \ref{wellposedness for linear system} for $p=2$. On the other hand, we can also use the linear theory established in \cite{danchin2014} to show the local well-posedness of \eqref{Lagrangian formulation of compressible pressureless flow}. But in \cite{danchin2014}, the constant $C$ in the linear estimate depends on $T$.

{\bf Step 3.} Fixed point argument. We shall perform the fixed point theorem in the Banach space $E_p(T)$ defined as
\begin{align*}
E_p(T)\stackrel{\mathrm{def}}{=}\left\{\bar{\u}\in C_b([0,T];\Dot{B}_{p,1}^{\frac{2}{p}-1})|\partial_t\bar{\u}\in L^1([0,T];\Dot{B}_{p,1}^{\frac{2}{p}-1}),\bar{\u}\in L^1([0,T];\Dot{B}_{p,1}^{\frac2p+1})\right\}
\end{align*}
endowed with the norm
\begin{align*}
\|\bar{\mathbf{u}}\|_{E_p}\stackrel{\mathrm{def}}{=}
\|\bar{\u}\|_{L_T^\infty(\Dot{B}_{p,1}^{\frac{2}{p}-1})}+\|\partial_t\bar{\u},\Delta \bar{\u}\|_{L_T^1(\Dot{B}_{p,1}^{\frac{2}{p}-1})}.
\end{align*}
We need the nonlinear estimates for $f(\bar{\u})$ when $\bar{\u}\in E_p(T)$ satisfying $\|\Delta \bar{\u}\|_{L_T^1(\Dot{B}_{p,1}^{\frac{2}{p}-1})}$ sufficiently small. So as in \cite{danchin2014}, we use the estimates in the appendix therein and product laws in Besov spaces to get
\begin{align}\label{nonlinear estimate 1}
\|f(\bar{\u})\|_{L_T^1(\Dot{B}_{p,1}^{\frac{2}{p}-1})}\le C\|\Delta\bar{\u}\|_{L_T^1(\Dot{B}_{p,1}^{\frac{2}{p}-1})}^2+CT\|\rho_0-1,m_0-1\|_{\Dot{B}_{p,1}^{\frac2p}}
\end{align}
Similarly, if both $\|\Delta \bar{\u}\|_{L_T^1(\Dot{B}_{p,1}^{\frac{2}{p}-1})}$ and $\|\Delta \bar{\w}\|_{L_T^1(\Dot{B}_{p,1}^{\frac{2}{p}-1})}$ are small, and if $T$ is also small, it holds that
\begin{align}\label{nonlinear estimate 2}
\|f(\bar{\u})-f(\bar{\w})\|_{L_T^1(\Dot{B}_{p,1}^{\frac{2}{p}-1})}\le \delta\|\bar{\u}-\bar{\w}\|_{E_p}
\end{align}
where $\delta$ is a small number.

Based on \eqref{nonlinear estimate 1}, \eqref{nonlinear estimate 2} and Theorem \ref{wellposedness for linear system}, the standard contraction mapping theorem guarantees a unique solution $\bar{\u}$ of \eqref{linearized} (hence \eqref{Lagrangian formulation of compressible pressureless flow}) in $E_p(T)$ provided $T$ is sufficiently small.

{\bf Step 4.} Back to the Euler coordinates. We can go back to the Euler coordinates through the inverse of $X$, where $X$ is defined by \eqref{trj2}. This gives the existence part of Proposition \ref{dingli1}. The uniqueness part can be proved by repeating Step 1-Step 3.

\bigskip

\section{The proof of  Theorem \ref{dingli2}}
In this section, we complete the proof of Theorem \ref{dingli2} in the following three subsections.
To find the hidden dissipation of the system \eqref{0mm2} and to avoid tedious calculations we may assume
that $\gamma=2$ (the case that $\gamma=1$ is much more easier), since the other cases can be essentially reduced to this case. We introduce an  unknown good function $\varphi$
as
\begin{align}\label{tade}
\varphi\stackrel{\mathrm{def}}{=}P +\frac12 m^2-\frac32.
\end{align}
Direct calculations show that
$(\varphi, \u)$ satisfies
\begin{eqnarray}\label{mm3}
\left\{\begin{aligned}
&\partial_t\varphi+3\div \u+\u\cdot\nabla \varphi+2\varphi\,\div \u=0,\\
&\partial_t\u + \u\cdot \nabla \u -  \mu\Delta \u - (\lambda+\mu)\nabla \div \u+\nabla \varphi =\mathbf{F}(a,\u,{\varphi}),\\
&\varphi|_{t=0}=\varphi_0\stackrel{\mathrm{def}}{=}a_0^2+2a_0 +\frac12 b_0^2+b_0,\quad \u|_{t=0}=\u_0,
\end{aligned}\right.
\end{eqnarray}
where
\begin{align}\label{}
\mathbf{F}(a,\u,{\varphi})\stackrel{\mathrm{def}}{=}&I(a)\nabla{\varphi}-I(a)(\mu\Delta \u+(\lambda+\mu)\nabla\div \u)\quad\hbox{with}\quad I(a)\stackrel{\mathrm{def}}{=}\frac{a}{1+a}.
\end{align}

Throughout we make the assumption that
\begin{equation}\label{eq:smalla}
\sup_{t\in\R_+,\, x\in\R^2} |a(t,x)|\leq \frac12
\end{equation}
which will enable us to use freely the composition estimate stated in Lemma \ref{fuhe}.
Note that as $\dot B^{\frac2p}_{p,1}(\R^2)\hookrightarrow L^\infty(\R^2),$ Condition \eqref{eq:smalla} will be ensured by the fact that the constructed solution has small norm in $\dot B^{\frac2p}_{p,1}(\R^2)$.

\subsection{Low-frequency estimates}
To study the coupling among $a,\varphi$ and $\q\u,$ it
is convenient to set  $$\q\u=-\Lambda^{-1}\nabla d.$$
 Since
${d}$ and $\mathbb Q \u =\nabla\Delta^{-1}\div \u$ can be converted into each other
by a zeroth-order homogeneous Fourier multiplier, it suffices to bound ${d}$ in order to
control $\mathbb Q \u$.  Now  one can  infer from \eqref{mm3} that
\begin{eqnarray}\label{ping3}
\left\{\begin{aligned}
&\partial_t\varphi+3\Lambda {d}=f_1, \\
&\partial_t{d} - 2\Delta {d}-\Lambda \varphi = f_2
\end{aligned}\right.
\end{eqnarray}
where
\begin{align*}
f_1\stackrel{\mathrm{def}}{=}&-\u\cdot\nabla \varphi-2\varphi\,\div \u,\quad
f_2\stackrel{\mathrm{def}}{=}\Lambda^{-1}\div(-(\u\cdot\nabla \u)+\mathbf{F}(a,\u,{\varphi})).
\end{align*}

In this subsection, we prove the following crucial lemma.
\begin{lemma}\label{dipinlemma}
 For any $t\ge0$, there holds that
\begin{align}\label{ping3-1}
&\norm{(\varphi^{\ell},{d}^{\ell})}_{\widetilde{L}^{\infty}_t(\dot{B}^{0}_{2,1})}
+\norm{(\varphi^{\ell},{d}^{\ell})}_{L^1_t(\dot{B}^{2}_{2,1})}\notag\\
&\quad\lesssim\norm{(\varphi^{\ell}_0,{d}^{\ell}_0)}_{\dot{B}^{0}_{2,1}}
+\norm{((f_{1})^{\ell},(f_{2})^{\ell})}_{L^1_t(\dot{B}^{0}_{2,1})}.
\end{align}
\end{lemma}

\begin{proof}
Let $k_0$ be some integer.
Setting
$f_k=\dot{\Delta}_kf$, applying the operator $\dot{\Delta}_k\dot{S}_{k_0}$ to the  equations in \eqref{ping3}, then multiplying
$(\ref{ping3})_1$ by $\varphi_{k}^{\ell}/3$, $(\ref{ping3})_2$ by ${d}_{k}^{\ell}$, respectively, we obtain
\begin{align}\label{ping3+1}
&\frac12\frac{d}{dt}\Big(\norm{\varphi_{k}^{\ell}}^2_{L^2}/3+\norm{{d}_k^{\ell}}^2_{L^2}\Big)+2\norm{\Lambda {d}_k^{\ell}}^2_{L^2}=\big\langle(f_{1})^{\ell}_k,\varphi^{\ell}_{k}/3\big\rangle+\big\langle(f_{2})^{\ell}_k ,{d}_k^{\ell}\big\rangle
\end{align}
where we have used the following cancellation
\begin{align}\label{}
\big\langle3\La{d}_k^{\ell},\varphi_k^{\ell}/3\big\rangle-\big\langle\La \varphi_k^{\ell},{d}_k^{\ell}\big\rangle=0.
\end{align}
To capture the dissipation of $\varphi $, we need to consider the time derivative of the mixed terms involved in $\big\langle{d}_k^{\ell},\Lambda \varphi_{k}^{\ell}\big\rangle$
\begin{align}\label{ping4}
&-\frac{d}{dt}\big\langle{d}_k^{\ell},\Lambda \varphi_{k}^{\ell}\big\rangle+\norm{\Lambda \varphi_{k}^{\ell}}^2_{L^2}-3\norm{\Lambda {d}_k^{\ell}}^2_{L^2}\notag\\
&\quad=-2\big\langle\Delta {d}_k^{\ell}, \Lambda \varphi_{k}^{\ell}\big\rangle-\big\langle(f_{2})^{\ell}_k ,\Lambda \varphi_{k}^{\ell}\big\rangle-\big\langle\Lambda (f_{1})^{\ell}_k ,{d}_k^{\ell}\big\rangle.
\end{align}
To eliminate the highest order terms on the right-hand sides of \eqref{ping4},
we next estimate $\norm{\Lambda \varphi_{k}^{\ell}}^2_{L^2}$.
 From \eqref{ping3}, we have
\begin{align}\label{aa}
\partial_t\Lambda \varphi_{k}^{\ell}+3\Lambda^2 {d}_{k}^{\ell}=\Lambda(f_1)_{k}^{\ell}.
\end{align}
Testing (\ref{aa}) by $2\Lambda \varphi_{k}^{\ell}/3$ yields
\begin{align}\label{ping6}
\frac13\frac{d}{dt}\norm{\Lambda \varphi_{k}^{\ell}}^2_{L^2}=
\big\langle2\Delta {d}_k^{\ell}, \Lambda \varphi_{k}^{\ell}\big\rangle+\big\langle(f_{1})^{\ell}_k ,2\Lambda^2 \varphi_{k}^{\ell}/3\big\rangle.
\end{align}
Denote
$$\mathcal{L}^2_k\stackrel{\mathrm{def}}{=}3\norm{\varphi_{k}^{\ell}}^2_{L^2}
+9\norm{{d}_k^{\ell}}^2_{L^2}-\big\langle{d}_k^{\ell},\Lambda \varphi_{k}^{\ell}\big\rangle
+\frac23\norm{\Lambda \varphi_{k}^{\ell}}^2_{L^2}.$$
Summing up $\eqref{ping3+1}\times9$, \eqref{ping4}, and $\eqref{ping6}$,  we obtain
\begin{align}\label{ping8}
&\frac12\frac{d}{dt}\mathcal{L}^2_k+\frac{33}{2}\norm{\Lambda {d}_k^{\ell}}^2_{L^2}+\frac12\norm{\Lambda \varphi_{k}^{\ell}}^2_{L^2}\notag\\
&\quad=3\big\langle(f_{1})^{\ell}_k ,a^{\ell}_{k}\big\rangle+9\big\langle(f_{2})^{\ell}_k ,{d}_k^{\ell} \big\rangle-\big\langle(f_{2})^{\ell}_k ,\Lambda \varphi_{k}^{\ell}\big\rangle-\big\langle\Lambda (f_{1})^{\ell}_k ,{d}_k^{\ell}\big\rangle+\frac23\big\langle(f_{1})^{\ell}_k ,\Lambda^2 \varphi_{k}^{\ell}\big\rangle.
\end{align}
It's straightforward to  deduce from the low-frequency cut-off and Young's inequality that
\begin{align*}
\mathcal{L}^2_k\thickapprox\norm{(\varphi_{k}^{\ell},\Lambda \varphi^{\ell}_k, {d}^{\ell}_k )}^2_{L^2}\thickapprox\norm{(\varphi_{k}^{\ell}, {d}^{\ell}_k)}^2_{L^2},
\end{align*}
which leads to
\begin{align}\label{ping9}
\frac12\frac{d}{dt}\mathcal{L}^2_k+2^{2k}\mathcal{L}^2_k\lesssim\norm{((f_{1})^{\ell}_k,(f_{2})^{\ell}_k )}_{L^2}\mathcal{L}_k.
\end{align}
 Dividing by $\mathcal{L}_k$ formally on both hand sides of \eqref{ping9}, and then integrating from $0$ to $t$, we finally get desired estimate \eqref{ping3-1} by summing up over $k\leq k_0$.
This proves the lemma.
\end{proof}

\vskip .1in
From Lemma \ref{dipinlemma} and the definitions of $\varphi$, $d$, $f_1$ and $f_2$,  the low frequency part of $(\varphi,\q \u)$ can be bounded by
\begin{align}\label{ping11}
&\norm{(\varphi^{\ell},\q \u^{\ell})}_{\widetilde{L}^{\infty}_t(\dot{B}^{0}_{2,1})}
+\norm{(\varphi^{\ell},\q \u^{\ell})}_{L^1_t(\dot{B}^{2}_{2,1})}\notag\\
&\quad\lesssim\norm{(\varphi^{\ell}_0,\q \u^{\ell}_0)}_{\dot{B}^{0}_{2,1}}
+\norm{(\u\cdot\nabla\varphi)^{\ell}}_{L^1_t(\dot{B}^{0}_{2,1})}+\norm{(\varphi\div\u)^{\ell}}_{L^1_t(\dot{B}^{0}_{2,1})}\notag\\
&\quad\quad
+\norm{(\u\cdot\nabla \u)^{\ell}}_{L^1_t(\dot{B}^{0}_{2,1})}
+\norm{(\mathbf{F}(a,\u,{\varphi})^{\ell}}_{L^1_t(\dot{B}^{0}_{2,1})}.
\end{align}
In the following, we estimate successively each of  terms on the right hand side of \eqref{ping11}.
To simplify the writing, we introduce the following notation:
\begin{align*}
&\mathcal{E}_\infty(t)\stackrel{\mathrm{def}}{=}\norm{(\varphi,\q\u)^\ell}_{\dot{B}_{2,1}^{0}}
+\norm{\varphi^h}_{\dot{B}_{p,1}^{\frac {2}{p}}}
+\norm{(\p\u,\q\u^h)}_{\dot{B}_{p,1}^{\frac {2}{p}-1}},
\nonumber\\
&\mathcal{E}_1(t)\stackrel{\mathrm{def}}{=}\norm{(\varphi,\q\u)^\ell}_{\dot{B}_{2,1}^{2}}
+\norm{\varphi^h}_{\dot{B}_{p,1}^{\frac {2}{p}}}
+\norm{(\p\u,\q\u^h)}_{\dot{B}_{p,1}^{\frac {2}{p}+1}}.
\end{align*}
First of all, in view of Lemma \ref{daishu}, there holds
\begin{align}\label{ping13}
&\norm{(\u\cdot\nabla\varphi)^{\ell}}_{\dot{B}^{0}_{2,1}}+\norm{(\varphi\div\u)^{\ell}}_{\dot{B}^{0}_{2,1}}\notag\\
&\quad\lesssim \norm{\u}_{\dot{B}^{\frac {2}{p}}_{p,1}}(\norm{\varphi^{\ell}}_{\dot{B}^{1}_{2,1}}+\norm{\varphi^h}_{\dot{B}^{\frac {2}{p}}_{p,1}}
)+\norm{\varphi}_{\dot{B}^{\frac {2}{p}}_{p,1}}
(\norm{\q\u^{\ell}}_{\dot{B}^{1}_{2,1}}+\norm{\q\u^h}_{\dot{B}^{\frac {2}{p}}_{p,1}})\notag\\
&\quad\lesssim \norm{\p\u}^2_{\dot{B}^{\frac {2}{p}}_{p,1}}+\norm{\q\u^{\ell}}^2_{\dot{B}^{1}_{2,1}}+\norm{\q\u^h}^2_{\dot{B}^{\frac {2}{p}}_{p,1}}
+\norm{\varphi^{\ell}}^2_{\dot{B}^{1}_{2,1}}+\norm{\varphi^h}^2_{\dot{B}^{\frac {2}{p}}_{p,1}}
\notag\\
&\quad\lesssim\norm{\p\u}_{\dot{B}^{\frac {2}{p}-1}_{p,1}}\norm{\p\u}_{\dot{B}^{\frac {2}{p}+1}_{p,1}}+\norm{\varphi^h}^2_{\dot{B}^{\frac {2}{p}}_{p,1}}\nn\\
&\qquad+ \Big(\norm{(\varphi^{\ell},\q\u^{\ell})}_{\dot{B}^{0}_{2,1}}+\norm{\q\u^h}_{\dot{B}^{\frac {2}{p}-1}_{p,1}}\Big)
\Big(\norm{(\varphi^{\ell},\q\u^{\ell})}_{\dot{B}^{2}_{2,1}}+\norm{\q\u^h}_{\dot{B}^{\frac {2}{p}+1}_{p,1}}\Big)\nn\\
&\quad\lesssim\e(t)\ee(t).
\end{align}
Next, to bound $\norm{(\u\cdot\nabla \u)^{\ell}}_{\dot{B}^{0}_{2,1}}$, we obtain from the decomposition $\u=\p\u+\q\u$  and Lemma \ref{daishu} that
\begin{align}\label{ping12}
\norm{(\u\cdot\nabla \u)^{\ell}}_{\dot{B}^{0}_{2,1}}
\lesssim& \norm{(\p\u\cdot\nabla \u)^{\ell}}_{\dot{B}^{0}_{2,1}}+\norm{(\q\u\cdot\nabla \u)^{\ell}}_{\dot{B}^{0}_{2,1}}\nn\\
\lesssim&\norm{\p\u}_{\dot{B}^{\frac {2}{p}-1}_{p,1}}\norm{\nabla \u}_{\dot{B}^{\frac {2}{p}}_{p,1}}+(\norm{\q\u^{\ell}}_{\dot{B}^{0}_{2,1}}+\norm{\q\u^h}_{\dot{B}^{\frac {2}{p}-1}_{p,1}})
\norm{\nabla \u}_{\dot{B}^{\frac {2}{p}}_{p,1}}.
\end{align}
Due to
\begin{align*}
\norm{\nabla \u}_{\dot{B}^{\frac {2}{p}}_{p,1}}\lesssim\norm{\p\u}_{\dot{B}^{\frac {2}{p}+1}_{p,1}}
+\norm{\q\u^{\ell}}_{\dot{B}^{2}_{2,1}}+\norm{\q\u^h}_{\dot{B}^{\frac {2}{p}+1}_{p,1}}\lesssim\ee(t),
\end{align*}
 we infer from \eqref{ping12} that
\begin{align}\label{ping12+1}
\norm{(\u\cdot\nabla \u)^{\ell}}_{\dot{B}^{0}_{2,1}}
\lesssim\e(t)\ee(t).
\end{align}
We now turn to bound the terms involving composition functions in $ \mathbf{F}(a,\u,{\varphi})$.
Keeping in mind that  $$I(a)=a-aI(a),$$  we first use Lemmas \ref{daishu} and \ref{fuhe} to get
\begin{align}\label{zijie1}
\norm{(I(a))^{\ell}}_{\dot{B}^{0}_{2,1}}
\lesssim&\norm{a^{\ell}}_{\dot{B}^{0}_{2,1}}+\norm{(aI(a))^{\ell}}_{\dot{B}^{0}_{2,1}}\nn\\
\lesssim&\norm{a^{\ell}}_{\dot{B}^{0}_{2,1}}+ \norm{I(a)}_{\dot{B}^{\frac {2}{p}}_{p,1}}(\norm{a^{\ell}}_{\dot{B}^{0}_{2,1}}
+\norm{a^h}_{\dot{B}^{\frac {2}{p}-1}_{p,1}})\notag\\
\lesssim&\norm{a^{\ell}}_{\dot{B}^{0}_{2,1}}+ \norm{a}_{\dot{B}^{\frac {2}{p}}_{p,1}}(\norm{a^{\ell}}_{\dot{B}^{0}_{2,1}}
+\norm{a^h}_{\dot{B}^{\frac {2}{p}}_{p,1}})\notag\\
\lesssim&\norm{a^{\ell}}_{\dot{B}^{0}_{2,1}}+ (\norm{a^{\ell}}_{\dot{B}^{0}_{2,1}}
+\norm{a^h}_{\dot{B}^{\frac {2}{p}}_{p,1}})^2\nn\\
\lesssim&(1+\e(t))\e(t).
\end{align}
Similarly, we can infer from Lemma \ref{law} and Lemma \ref{fuhe} that
\begin{align}\label{zijie2}
\|I(a)\|_{\dot{B}_{p,1}^{\frac {2}{p}-1}}
\lesssim&\|a\|_{\dot{B}_{p,1}^{\frac {2}{p}-1}}+\|aI(a)\|_{\dot{B}_{p,1}^{\frac {2}{p}-1}}\nn\\
\lesssim&(\norm{a^{\ell}}_{\dot{B}^{0}_{2,1}}
+\norm{a^h}_{\dot{B}^{\frac {2}{p}}_{p,1}})+\|a\|_{\dot{B}_{p,1}^{\frac {2}{p}-1}}\|I(a)\|_{\dot{B}_{p,1}^{\frac {2}{p}}}\nn\\
\lesssim&(\norm{a^{\ell}}_{\dot{B}^{0}_{2,1}}
+\norm{a^h}_{\dot{B}^{\frac {2}{p}}_{p,1}})+\|a\|_{\dot{B}_{p,1}^{\frac {2}{p}-1}}\|a\|_{\dot{B}_{p,1}^{\frac {2}{p}}}\nn\\
\lesssim&(1+\norm{a^{\ell}}_{\dot{B}^{0}_{2,1}}
+\norm{a^h}_{\dot{B}^{\frac {2}{p}}_{p,1}})(\norm{a^{\ell}}_{\dot{B}^{0}_{2,1}}
+\norm{a^h}_{\dot{B}^{\frac {2}{p}}_{p,1}})\nn\\
\lesssim&(1+\e(t))\e(t).
\end{align}
Now, for the first term $\norm{(I(a)\nabla {\varphi})^\ell}_{\dot{B}^{0}_{2,1}}$ in $ \mathbf{F}(a,\u,{\varphi})$, in view of the fact that ${\varphi}={\varphi}^\ell+{\varphi}^h$, we can write
\begin{align}\label{zijie3}
\norm{(I(a)\nabla {\varphi})^\ell}_{\dot{B}^{0}_{2,1}}\lesssim\norm{(I(a)\nabla {\varphi}^\ell)^{\ell}}_{\dot{B}^{0}_{2,1}}+\norm{(I(a)\nabla {\varphi}^h)^{\ell}}_{\dot{B}^{0}_{2,1}}.
\end{align}
Thanks to Lemma \ref{daishu} again, we have
\begin{align*}
\norm{(I(a)\nabla {\varphi}^\ell)^{\ell}}_{\dot{B}^{0}_{2,1}}
\lesssim& \norm{\nabla {\varphi}^\ell}_{\dot{B}^{\frac {2}{p}}_{p,1}}(\norm{(I(a))^{\ell}}_{\dot{B}^{0}_{2,1}}
+\norm{(I(a))^h}_{\dot{B}^{\frac {2}{p}-1}_{p,1}})\notag\\
\lesssim& \norm{\varphi^\ell}_{\dot{B}^{2}_{2,1}}(\norm{(I(a))^{\ell}}_{\dot{B}^{0}_{2,1}}
+\norm{(I(a))^h}_{\dot{B}^{\frac {2}{p}-1}_{p,1}})
\end{align*}
which combines \eqref{zijie1} and \eqref{zijie2} leads to
\begin{align}\label{zijie6}
\norm{(I(a)\nabla {\varphi}^\ell)^{\ell}}_{\dot{B}^{0}_{2,1}}
\lesssim& \norm{ {\varphi}^\ell}_{\dot{B}^{2}_{2,1}}(1+\e(t))\e(t).
\end{align}
For the term
$\norm{(I(a)\nabla {\varphi}^h)^{\ell}}_{\dot{B}^{0}_{2,1}}$ in \eqref{zijie3}, we use Bony's decomposition to write
\begin{align}\label{zijie7}
\dot S_{j_0+1} (I(a)\nabla {\varphi}^h)=&\dot{T}_{I(a)}\dot S_{j_0+1}  \nabla {\varphi}^h+[\dot S_{j_0+1} ,\dot{T}_{I(a)}]\nabla {\varphi}^h\nn\\
&+\dot S_{j_0+1} \bigl(\dot{T}_{\nabla {\varphi}^h} I(a)+ \dot{R}(I(a),\nabla {\varphi}^h)\bigr)
.
\end{align}
Applying Lemma \ref{fangji}, there holds
\begin{align}\label{zijie8}
\|\dot{T}_{I(a)}\dot S_{j_0+1}  \nabla {\varphi}^h\|_{\dot{B}_{2,1}^{0}}\lesssim&\|I(a)\|_{\dot{B}_{\infty,\infty}^{-1}}\|\dot S_{j_0+1}  \nabla {\varphi}^h\|_{\dot{B}_{2,1}^{1}}
\lesssim\|I(a)\|_{\dot{B}_{p,1}^{\frac {2}{p}-1}}\|\varphi\|^\ell_{\dot{B}_{2,1}^{2}},
\end{align}
from which and \eqref{zijie2}, we can further get
\begin{align}\label{zijie9}
\|\dot{T}_{I(a)}\dot S_{j_0+1}  \nabla {\varphi}^h\|_{\dot{B}_{2,1}^{0}}
\lesssim(1+\e(t))\e(t)\|\varphi\|^\ell_{\dot{B}_{2,1}^{2}}.
\end{align}
The last two terms in \eqref{zijie7} can be estimated the same as   \eqref{D116} and \eqref{D11}   so that
\begin{align}\label{zijie10}
&\|\dot S_{j_0+1} \bigl(\dot{T}_{\nabla {\varphi}^h} I(a)+ \dot{R}(I(a),\nabla {\varphi}^h)\bigr)\|_{\dot{B}_{2,1}^{0}}
+
\|[\dot S_{j_0+1} ,\dot{T}_{I(a)}]\nabla {\varphi}^h\|_{\dot{B}_{2,1}^{0}}\nn\\
&\quad\lesssim\|\nabla {\varphi}^h\|_{\dot{B}_{p,1}^{\frac {2}{p}-1}}\|I(a)\|_{\dot{B}_{p,1}^{\frac {2}{p}}} \nn\\
&\quad\lesssim\| {\varphi}^h\|_{\dot{B}_{p,1}^{\frac {2}{p}}}\|a\|_{\dot{B}_{p,1}^{\frac {2}{p}}} \nn\\
&\quad\lesssim(\norm{a^{\ell}}_{\dot{B}^{0}_{2,1}}
+\norm{a^h}_{\dot{B}^{\frac {2}{p}}_{p,1}})\| {\varphi}^h\|_{\dot{B}_{p,1}^{\frac {2}{p}}}\nn\\
&\quad\lesssim\e(t)\| {\varphi}^h\|_{\dot{B}_{p,1}^{\frac {2}{p}}}
\end{align}
this together with \eqref{zijie9} give rise to
\begin{align}\label{zijie11}
\norm{(I(a)\nabla {\varphi}^h)^{\ell}}_{\dot{B}^{0}_{2,1}}
\lesssim(1+\e(t))\e(t)(\|\varphi\|^\ell_{\dot{B}_{2,1}^{2}}+\| {\varphi}^h\|_{\dot{B}_{p,1}^{\frac {2}{p}}}).
\end{align}

Plugging \eqref{zijie6} and \eqref{zijie11}  into \eqref{zijie3} yields
\begin{align}\label{zijie12}
\norm{(I(a)\nabla {\varphi})}_{\dot{B}^{0}_{2,1}}\lesssim(1+\e(t))\e(t)(\|\varphi\|^\ell_{\dot{B}_{2,1}^{2}}+\| {\varphi}^h\|_{\dot{B}_{p,1}^{\frac {2}{p}}}).
\end{align}
For the last term in $ \mathbf{F}(a,\u,{\varphi})$, as we set $\p\u$  in the $L^p$ type spaces, we cannot use Lemma \ref{daishu} directly to bound this term. For an integer $j_0\ge 0$,
we use Bony's decomposition to rewrite this term into
\begin{align}\label{ping17}
\dot S_{j_0+1}\q (I(a)\Delta \u)=&\dot S_{j_0+1}\q \bigl(\dot{T}_{\Delta \u} I(a)+ \dot{R}(\Delta \u,I(a))\bigr)\nonumber\\
&+\dot{T}_{I(a)}\Delta\dot S_{j_0+1}\q  \u
+[\dot S_{j_0+1}\q ,\dot{T}_{I(a)}]\Delta \u.
\end{align}
The first term can be bounded by Lemmas    \ref{fangji} and \ref{fuhe},
\begin{align}\label{ping18}
&\norm{(\dot S_{j_0+1}\q \bigl(\dot{T}_{\Delta \u} I(a)+ \dot{R}(\Delta \u,I(a))\bigr)  )^\ell }_{\dot B^{0}_{2,1}}\nonumber\\
&\quad\lesssim \norm{I(a)}_{\dot B^{\frac {2}{p}}_{p,1}}\norm{\Delta \u}_{\dot B^{\frac {2}{p}-1}_{p,1}}
\lesssim \norm{a}_{\dot B^{\frac {2}{p}}_{p,1}}\norm{ \u}_{\dot B^{\frac {2}{p}+1}_{p,1}}\nonumber\\
&\quad\lesssim (\norm{a^\ell}_{\dot B^{0}_{2,1}}+\norm{a^h}_{\dot B^{\frac {2}{p}}_{p,1}})(\norm{\q \u^\ell} _{\dot B^{2}_{2,1}}+\norm{(\p \u,\q \u^h)}_{\dot B^{\frac {2}{p}+1}_{p,1}} ).
\end{align}
Similarly, we have
\begin{align}\label{ping19}
\norm{(\dot{T}_{I(a)}\Delta\dot S_{j_0+1}\q  \u  )^\ell  }_{\dot B^{0}_{2,1}}
\lesssim& \norm{I(a)}_{L^\infty}\norm{\Delta\dot S_{j_0+1}\q  \u}_{\dot B^{0}_{2,1}}
\nonumber\\
\lesssim& \norm{a}_{\dot B^{\frac {2}{p}}_{p,1}}\norm{ \q \u^\ell}_{\dot B^{2}_{2,1}}\nonumber\\
\lesssim& (\norm{a^\ell}_{\dot B^{0}_{2,1}}+\norm{a^h}_{\dot B^{\frac {2}{p}}_{p,1}})\norm{\q \u^\ell} _{\dot B^{2}_{2,1}}.
\end{align}
The commutator term is estimated by using  Lemma  6.1 in \cite{helingbing} that
\begin{align}\label{ping20}
\norm{[\dot S_{j_0+1}\q ,T_{I(a)}]\Delta \u}_{\dot B^{0}_{2,1}}
\lesssim& \norm{\nabla I(a)}_{\dot B^{\frac{2}{p^*}-1}_{p^*,1}}\norm{\nabla^2\u}_{ \dot B^{\frac {2}{p}-1}_{p,1}},\quad\quad\left(\frac{1}{p^*}+\frac1p=\frac12\right), \nonumber\\
\quad\lesssim& \norm{\nabla I(a)}_{\dot B^{\frac {2}{p}-1}_{p,1}}\norm{ \u}_{\dot B^{\frac {2}{p}+1}_{p,1}} \nonumber\\
\lesssim& \norm{a}_{\dot B^{\frac {2}{p}}_{p,1}}\norm{ \u}_{\dot B^{\frac {2}{p}+1}_{p,1}}\nonumber\\
\quad\lesssim& (\norm{a^\ell}_{\dot B^{0}_{2,1}}+\norm{a^h}_{\dot B^{\frac {2}{p}}_{p,1}})(\norm{\q \u^\ell} _{\dot B^{2}_{2,1}}+\norm{(\p \u,\q \u^h)}_{\dot B^{\frac {2}{p}+1}_{p,1}} ),
\end{align}
where  we have used the embedding  $\dot B^{\frac {2}{p}-1}_{p,1}(\R^2)\hookrightarrow \dot B^{\frac{2}{p^*}-1}_{p^*,1}(\R^2), \ 2\le p<4.$

The term $I(a)\nabla\div \u$ can be estimated in a similar manner. As a result, we have
\begin{align}\label{ping22}
\norm{(I(a)(\Delta \u+\nabla\div \u))^{\ell}}_{\dot{B}^{0}_{2,1}}
\lesssim\e(t)\ee(t).
\end{align}
Plugging \eqref{ping13}, \eqref{ping12+1}, \eqref{zijie12}, and \eqref{ping22} into \eqref{ping11} gives
\begin{align}\label{ping24}
&\norm{(\varphi^{\ell},\q\u^{\ell})}_{\widetilde{L}^{\infty}_t(\dot{B}^{0}_{2,1})}
+\norm{(\varphi^{\ell},\q\u^{\ell})}_{L^1_t(\dot{B}^{2}_{2,1})}\notag\\
&\quad\lesssim\norm{(\varphi^{\ell}_0,\q\u^{\ell}_0)}_{\dot{B}^{0}_{2,1}}
+\int^t_0(1+\e(\tau))\e(\tau)\ee(\tau)\,d\tau.
\end{align}
Finally, we shall derive the bound of $\norm{a^{\ell}}_{\widetilde{L}^{\infty}_t(\dot{B}^{0}_{2,1})}.$
Due to the appearance of the  term $\div \u$ in the first equation of \eqref{333mm3}, we cannot obtain the bound
$\norm{a^{\ell}}_{\widetilde{L}^{\infty}_t(\dot{B}^{0}_{2,1})}$ directly.
To break the  barrier, we define
\begin{align}\label{ggug1}
\delta\stackrel{\mathrm{def}}{=}\varphi-3a
\end{align}
which satisfies the following  transport equation
\begin{align}\label{ggug12}
\partial_t\delta+\u\cdot\nabla \delta+\delta\,\div \u+\varphi\,\div \u=0.
\end{align}
Now applying $\ddj$ to the above equation and using a commutator's argument give rise to
\begin{align*}
\partial_t\ddj{\delta}+\u\cdot\nabla \ddj \delta+[\ddj,\u\cdot\nabla] \delta+\ddj(\delta\div \u)+\ddj(\varphi\div \u)=0.
\end{align*}
Taking  $L^2$ inner product of the resulting equation with $\dot{\Delta}_{j}\delta$,
applying the H\"older inequality and
integrating the resultant inequality over $[0, t]$, then   summing up  $j\le j_0$, we arrive at
\begin{align}\label{you2}
\|\delta^\ell\|_{ \widetilde{L}_t^\infty(\dot B^{0}_{2,1})}
\lesssim&\|\delta^\ell_0\|_{\dot B^{0}_{2,1}}
+ \|(\delta\div\u)^\ell\|_{L^1_t(\dot B^{0}_{2,1})}+ \|(\varphi\div\u)^\ell\|_{L^1_t(\dot B^{0}_{2,1})}\nn\\
&+\int_0^t\|\div \u\|_{L^\infty}\|\delta^\ell\|_{\dot B^{0}_{2,1}}\,d\tau+\int_0^t\sum_{j\le j_0}\|[\ddj,\u\cdot\nabla]\delta\|_{L^2}\,d\tau.
\end{align}
By Lemma \ref{daishu}, there holds
\begin{align}\label{you3}
\|(\delta\div\u)^\ell\|_{\dot B^{0}_{2,1}}
\lesssim&(\|\delta^\ell\|_{\dot B^{0}_{2,1}}+\|\delta^h\|_{\dot B^{\frac  2p-1}_{p,1}})\|\div\u\|_{\dot B^{\frac  2p}_{p,1}}\nn\\
\lesssim&(\|(a^\ell,\varphi^\ell)\|_{\dot B^{0}_{2,1}}+\|(a^h,\varphi^h)\|_{\dot B^{\frac  2p}_{p,1}})(\|\q\u^\ell\|_{\dot B^{2}_{2,1}}+\|\q\u^h\|_{\dot B^{\frac  2p+1}_{p,1}}).
\end{align}
Similarly,
\begin{align}\label{you3232}
\|(\varphi\div\u)^\ell\|_{\dot B^{0}_{2,1}}
\lesssim&(\|\varphi^\ell\|_{\dot B^{0}_{2,1}}+\|\varphi^h\|_{\dot B^{\frac  2p}_{p,1}})(\|\q\u^\ell\|_{\dot B^{2}_{2,1}}+\|\q\u^h\|_{\dot B^{\frac  2p+1}_{p,1}}).
\end{align}

With the aid of the embedding relation ${\dot B^{\frac  2p}_{p,1}}(\R^2)\hookrightarrow L^\infty(\R^2)$ and Lemma \ref{jiaohuanzi}, we can bound the forth term on the right hand side of \eqref{you2} as
\begin{align}\label{you5}
\|\div \u\|_{L^\infty}\|\delta^\ell\|_{\dot B^{0}_{2,1}}
\lesssim&\|\delta^\ell\|_{\dot B^{0}_{2,1}}\|\q\u\|_{\dot B^{\frac  2p+1}_{p,1}}\nn\\
\lesssim&(\|(a^\ell,\varphi^\ell)\|_{\dot B^{0}_{2,1}}+\|(a^h,\varphi^h)\|_{\dot B^{\frac  2p}_{p,1}})(\|\q\u^\ell\|_{\dot B^{2}_{2,1}}+\|\q\u^h\|_{\dot B^{\frac  2p+1}_{p,1}}).
\end{align}
The last term in \eqref{you2} can be bounded by a  similarly derivation of  (4.9)  in  \cite{huangjingchi} that
\begin{align}\label{you6}
\sum_{j\le j_0}\|[\ddj,\u\cdot\nabla]\delta\|_{L^2}
\lesssim&(\big\|\q\u^\ell\big\|_{\dot{B}^{2}_{2,1}}+\big\|(\p\u,\u^h)\big\|_{\dot{B}^{\frac{2}{p}+1}_{p,1}})
(\|(a^\ell,\varphi^\ell)\|_{\dot B^{0}_{2,1}}+\|(a^h,\varphi^h)\|_{\dot B^{\frac  2p}_{p,1}}).
\end{align}
Taking \eqref{you3}--\eqref{you6} into \eqref{you2}, we obtain
\begin{align}\label{youjia2}
\|\delta^\ell\|_{ \widetilde{L}_t^\infty(\dot B^{0}_{2,1})}
\lesssim&\|(a^\ell_0,\varphi^\ell_0)\|_{\dot B^{0}_{2,1}}+
\int^t_0\e(\tau)\ee(\tau)\,d\tau
\end{align}
which combines the definition $a=\frac13(\varphi-\delta)$ leads to
\begin{align}\label{youjia2222}
\|a^\ell\|_{ \widetilde{L}_t^\infty(\dot B^{0}_{2,1})}
\lesssim&\|\delta^\ell\|_{ \widetilde{L}_t^\infty(\dot B^{0}_{2,1})}+\|\varphi^\ell\|_{ \widetilde{L}_t^\infty(\dot B^{0}_{2,1})}\nn\\
\lesssim&
\|(a^\ell_0,\varphi^\ell_0)\|_{\dot B^{0}_{2,1}}+\|\varphi^\ell\|_{ \widetilde{L}_t^\infty(\dot B^{0}_{2,1})}+
\int^t_0\e(\tau)\ee(\tau)\,d\tau.
\end{align}
In the same manner, we can infer from forth equation of \eqref{333mm3} that
\begin{align}\label{youjia22232}
\|b^\ell\|_{ \widetilde{L}_t^\infty(\dot B^{0}_{2,1})}
\lesssim&
\|(b^\ell_0,\varphi^\ell_0)\|_{\dot B^{0}_{2,1}}+\|\varphi^\ell\|_{ \widetilde{L}_t^\infty(\dot B^{0}_{2,1})}+
\int^t_0\e(\tau)\ee(\tau)\,d\tau.
\end{align}
Consequently, combining with \eqref{ping24}, \eqref{youjia2222} and \eqref{youjia22232}, we finally arrive at
\begin{align}\label{ping246}
&\norm{(a^{\ell},b^{\ell},\q\u^{\ell})}_{\widetilde{L}^{\infty}_t(\dot{B}^{0}_{2,1})}
+\norm{(\varphi^{\ell},\q\u^{\ell})}_{L^1_t(\dot{B}^{2}_{2,1})}\notag\\
&\quad\lesssim\norm{(a^{\ell}_0,b^{\ell}_0,\varphi^{\ell}_0,\q\u^{\ell}_0)}_{\dot{B}^{0}_{2,1}}
+\int^t_0(1+\e(\tau))\e(\tau)\ee(\tau)\,d\tau.
\end{align}

\subsection{High-frequency estimates}
In this subsection, we shall introduce the so called  effective velocity to capture the damping effect of $\varphi$ in the high frequency part.
\subsubsection{Estimates for auxiliary unknowns}
 First, we infer from \eqref{mm3} that $(\varphi,\q\u)$ satisfies
\begin{eqnarray}\label{han1}
\left\{\begin{aligned}
&\partial_t{\varphi}+{3}\div \u=-\varphi\,\div \u-\div(\varphi\u),\\
&\partial_t\mathbb{Q}\u - 2\Delta \mathbb{Q}\u +\nabla {\varphi} =-\mathbb{Q}(\u\cdot\nabla \u)+\mathbb{Q}\mathbf{F}(a,\u,{\varphi}).
\end{aligned}\right.
\end{eqnarray}
Now, we define the  effective velocity ${\mathbf{G}}$ as follows
\begin{align}\label{han2}
{\mathbf{G}}\stackrel{\mathrm{def}}{=} \mathbb{Q}\u-\frac12\Delta^{-1}\nabla {\varphi}.
\end{align}
Then ${\mathbf{G}}$ satisfies
\begin{align}\label{han3}
\partial_t{\mathbf{G}}-2\Delta {\mathbf{G}}=\frac32{\mathbf{G}}+\frac32\Delta^{-1}\nabla {\varphi}+\frac12\Delta^{-1}\nabla(\varphi\,\div \u)+\frac12\mathbb{Q}({\varphi} \u)-\mathbb{Q}(\u,\nabla \u)+\mathbb{Q}\mathbf{F}(a,\u,{\varphi}).
\end{align}
Applying the heat estimate \eqref{heat} for the high frequencies of $\mathbf{G}$ only, we get
\begin{align}\label{han5}
&\norm{\mathbf{G}^h}_{ \widetilde{L}_t^\infty(\dot B^{\frac  2p-1}_{p,1})}+ \norm{\mathbf{G}^h}_{L^1_t(\dot B^{\frac  2p+1}_{p,1})}\nn\\
&\quad\lesssim \norm{\mathbf{G}_0^h}_{\dot B^{\frac  2p-1}_{p,1}}
+ \norm{\mathbf{G}^h}_{L^1_t(\dot B^{\frac  2p-1}_{p,1})}+ \norm{\varphi^h}_{L^1_t(\dot B^{\frac  2p-2}_{p,1})}+ \norm{(\varphi\div\u)^h}_{L^1_t(\dot B^{\frac  2p-2}_{p,1})}\nonumber\\
&\qquad+ \norm{\mathbb{Q}({\varphi} \u)^h}_{L^1_t(\dot B^{\frac  2p-1}_{p,1})}+\norm{\mathbb{Q}(\u,\nabla \u)^h}_{L^1_t(\dot B^{\frac  2p-1}_{p,1})}+\norm{\mathbb{Q}\mathbf{F}(a,\u,{\varphi})^h}_{L^1_t(\dot B^{\frac  2p-1}_{p,1})}.
\end{align}
The important point is that, owing to the high frequency cut-off at $|\xi|\sim 2^{j_0},$
$$
\norm{\mathbf{G}^h}_{L^1_t(\dot B^{\frac  2p-1}_{p,1})}\lesssim 2^{-2j_0}\norm{\mathbf{G}^h}_{L^1_t(\dot B^{\frac  2p+1}_{p,1})}
\quad\hbox{and}\quad
\norm{\varphi^h}_{L^1_t(\dot B^{\frac  2p-2}_{p,1})}\lesssim  2^{-2j_0}\norm{\varphi^h}_{L^1_t(\dot B^{\frac  2p}_{p,1})}.
$$
Hence, if $j_0$ is large enough then the term $\norm{\mathbf{G}^h}_{L^1_t(\dot B^{\frac  2p-1}_{p,1})}$ may be absorbed by the right hand side.

In view of \eqref{han2}, we have
 ${\varphi}$ satisfies
\begin{align}\label{han4}
\partial_t{\varphi}+\frac32{\varphi}+\u\cdot\nabla \varphi=-3\div {\mathbf{G}}-2{\varphi}\,\div \u.
\end{align}
Applying $\ddj$ to \eqref{han4} and using a commutator argument give rise to
\begin{align}\label{han4+1}
\partial_t\ddj{\varphi}+\frac32\ddj{\varphi}+\u\cdot\nabla \ddj \varphi=-[\ddj,\u\cdot\nabla] \varphi-3\ddj\div {\mathbf{G}}-2\ddj({\varphi}\,\div \u).
\end{align}
Taking  $L^2$ inner product of \eqref{han4+1} with $\frac1p|\dot{\Delta}_{j}\varphi|^{p-2}\dot{\Delta}_{j}\varphi$,
applying the H\"older inequality and
integrating the resultant inequality over $[0, t]$  lead to
\begin{align}\label{han12}
&\|\ddj \varphi(t)\|_{L^p}+\int_0^t\|\ddj \varphi\|_{L^p}\,d\tau\nonumber\\
&\quad\lesssim\|\ddj \varphi_0\|_{L^p}
+\frac 1p\int_0^t\norm{\div \u}_{L^\infty}\|\ddj \varphi\|_{L^p}\,d\tau
\nonumber\\
&\quad\quad+\int_0^t\|[\ddj,\u\cdot\nabla] \varphi\|_{L^p}\,d\tau+\int_0^t\|\ddj\div \mathbf{G}\|_{L^p}\,d\tau+\int_0^t\|\ddj(\varphi\,\div  \u)\|_{L^p}\,d\tau
\end{align}
from which we can further get
\begin{align}\label{han6}
&\norm{\varphi^h}_{\widetilde{L}^\infty_t(\dot B^{\frac  2p}_{p,1})}+\frac32\norm{\varphi^h}_{ L^1_t(\dot B^{\frac  2p}_{p,1})}\notag \\
&\quad\lesssim \norm{\varphi_0^h}_{\dot B^{\frac  2p}_{p,1}}+3\norm{\mathbf{G}^h}_{L^1_t(\dot B^{\frac  2p+1}_{p,1})}+\int_0^t\norm{ \nabla\u}_{\dot B^{\frac  2p}_{p,1}}\norm{\varphi}_{\dot B^{\frac  2p}_{p,1}}\,d\tau.
\end{align}
Multiplying \eqref{han5} by a suitable large constant and adding to \eqref{han6}, we obtain
\begin{align}\label{han7}
&\norm{{\mathbf{G}}^h}_{\widetilde{L}^{\infty}_t(\dot{B}^{\frac {2}{p}-1}_{p,1})}+\norm{{\varphi}^h}_{\widetilde{L}^{\infty}_t(\dot{B}^{\frac {2}{p}}_{p,1})}
+\norm{{\mathbf{G}}^h}_{{L}^{1}_t(\dot{B}^{\frac {2}{p}+1}_{p,1})}+\norm{{\varphi}^h}_{{L}^{1}_t(\dot{B}^{\frac {2}{p}}_{p,1})}\notag\\
&\quad\lesssim \norm{{\mathbf{G}}^h_{0}}_{\dot{B}^{\frac {2}{p}-1}_{p,1}}+\norm{{\varphi}^h_{0}}_{\dot{B}^{\frac {2}{p}}_{p,1}}
+\int_0^t\norm{\nabla \u}_{\dot B^{\frac  2p}_{p,1}}\norm{\varphi}_{\dot B^{\frac  2p}_{p,1}}\,d\tau\notag\\
&\quad\quad +\int^t_0(
\norm{\u\cdot\nabla \u}_{\dot{B}^{\frac {2}{p}-1}_{p,1}}+\norm{({\varphi} \u)^h}_{\dot{B}^{\frac {2}{p}-1}_{p,1}})\,d\tau
+\int^t_0\norm{(\mathbf{F}(a,\u,{\varphi}))^h}_{\dot{B}^{\frac {2}{p}-1}_{p,1}}\,d\tau.
\end{align}
\subsubsection{Recovering estimates for a, b}
In this subsection, we shall recover the estimates for $a,b$, and  $\u$.
On the one hand,
in view of ${\mathbf{G}}\stackrel{\mathrm{def}}{=} \mathbb{Q}\u-\frac12\Delta^{-1}\nabla {\varphi}$ and the embedding relation in the high frequency, there holds
\begin{align*}
\norm{\q \u^h}_{\widetilde{L}^{\infty}_t(\dot{B}^{\frac {2}{p}-1}_{p,1})}
\lesssim& \norm{\mathbf{G}^h}_{\widetilde{L}^{\infty}_t(\dot{B}^{\frac {2}{p}-1}_{p,1})}
+\norm{ \varphi^h}_{\widetilde{L}^{\infty}_t(\dot{B}^{\frac {2}{p}}_{p,1})},
\end{align*}
\begin{align*}
\norm{\q \u^h}_{{L}^{1}_t(\dot{B}^{\frac {2}{p}+1}_{p,1})}
\lesssim & \norm{\mathbf{G}^h}_{{L}^{1}_t(\dot{B}^{\frac {2}{p}+1}_{p,1})}
+\norm{ \varphi^h}_{{L}^{1}_t(\dot{B}^{\frac {2}{p}}_{p,1})}.
\end{align*}
As a result, we can rewrite \eqref{han7} into
\begin{align}\label{han21}
&\norm{\varphi^h}_{\widetilde{L}^{\infty}_t(\dot{B}^{\frac {2}{p}}_{p,1})}
+\norm{\q\u^h}_{\widetilde{L}^{\infty}_t(\dot{B}^{\frac {2}{p}-1}_{p,1})}+\norm{\varphi^h}_{{L}^{1}_t(\dot{B}^{\frac {2}{p}}_{p,1})}
+\norm{{\q\u}^h}_{{L}^{1}_t(\dot{B}^{\frac {2}{p}+1}_{p,1})}\nn\\
&\quad\lesssim \norm{{\varphi}^h_{0}}_{\dot{B}^{\frac {2}{p}}_{p,1}}+
\norm{{\q\u}^h_{0}}_{\dot{B}^{\frac {2}{p}-1}_{p,1}}
+\int_0^t\norm{\nabla \u}_{\dot B^{\frac  2p}_{p,1}}\norm{\varphi}_{\dot B^{\frac  2p}_{p,1}}\,d\tau\notag\\
&\quad\quad +\int^t_0(
\norm{\u\cdot\nabla \u}_{\dot{B}^{\frac {2}{p}-1}_{p,1}}+\norm{({\varphi} \u)^h}_{\dot{B}^{\frac {2}{p}-1}_{p,1}})\,d\tau
+\int^t_0\norm{(\mathbf{F}(a,\u,{\varphi}))^h}_{\dot{B}^{\frac {2}{p}-1}_{p,1}}\,d\tau.
\end{align}
Finally, we estimates the incompressible part of the velocity field.
Applying the operator $\p$ to the second equation of \eqref{mm3}, we find that $\p\u$ satisfies the heat equation
\begin{align}\label{ping1}
\partial_t {\p \u}-\Delta {\p \u}=-{\p (\u\cdot \nabla \u)}+\p \mathbf{F}(a,\u,{\varphi}) .
\end{align}
By Lemma \ref{heat}, we can get
\begin{align}\label{ping2}
&\norm{{\p \u}}_{\widetilde{L}^{\infty}_t(\dot{B}^{\frac {2}{p}-1}_{p,1})}
+\norm{{\p \u}}_{L^1_t(\dot{B}^{\frac {2}{p}+1}_{p,1})}
\nn\\
&\quad\lesssim\norm{\p \u_0}_{\dot{B}^{\frac {2}{p}-1}_{p,1}}
+\norm{(\u\cdot\nabla \u)}_{L^1_t(\dot{B}^{\frac {2}{p}-1}_{p,1})}+\norm{(\mathbf{F}(a,\u,{\varphi})}_{L^1_t(\dot{B}^{\frac {2}{p}-1}_{p,1})}.
\end{align}
Combining  \eqref{han21} with \eqref{ping2} gives
\begin{align}\label{keyaguji}
&\norm{\varphi^h}_{\widetilde{L}^{\infty}_t(\dot{B}^{\frac {2}{p}}_{p,1})}
+\norm{(\p\u,\q\u^h)}_{\widetilde{L}^{\infty}_t(\dot{B}^{\frac {2}{p}-1}_{p,1})}
\nn\\
&\qquad+\norm{\varphi^h}_{{L}^{1}_t(\dot{B}^{\frac {2}{p}}_{p,1})}
+\norm{(\p\u,\q\u^h)}_{{L}^{1}_t(\dot{B}^{\frac {2}{p}+1}_{p,1})}\nn\\
&\quad\lesssim \norm{{\varphi}^h_{0}}_{\dot{B}^{\frac {2}{p}}_{p,1}}+
\norm{(\p\u_0,{\q\u}^h_{0})}_{\dot{B}^{\frac {2}{p}-1}_{p,1}}
+\int_0^t\norm{\nabla \u}_{\dot B^{\frac  2p}_{p,1}}\norm{\varphi}_{\dot B^{\frac  2p}_{p,1}}\,d\tau\notag\\
&\quad\quad +\int^t_0(
\norm{\u\cdot\nabla \u}_{\dot{B}^{\frac {2}{p}-1}_{p,1}}+\norm{({\varphi} \u)^h}_{\dot{B}^{\frac {2}{p}-1}_{p,1}})\,d\tau
+\int^t_0\norm{\mathbf{F}(a,\u,{\varphi})}_{\dot{B}^{\frac {2}{p}-1}_{p,1}}\,d\tau.
\end{align}
We now bound the terms on the right hand side of \eqref{keyaguji}.
First, it's obvious that
\begin{align}\label{}
\norm{\nabla \u}_{\dot B^{\frac  2p}_{p,1}}\norm{\varphi}_{\dot B^{\frac  2p}_{p,1}}
\lesssim&\big(
\|\q\u^\ell\|_{\dot{B}^{2}_{2,1}}+\|(\p\u,\q\u^h)\|_{\dot{B}^{\frac {2}{p}+1}_{p,1}})(\|\varphi^\ell\|_{\dot{B}^{ 0}_{2,1}}+\|\varphi^h\|_{\dot{B}^{\frac {2}{p}}_{p,1}}\big).
\end{align}
Then, according to Lemma \ref{law}, there holds
\begin{align}\label{}
\norm{\u\cdot\nabla \u}_{\dot{B}^{\frac {2}{p}-1}_{p,1}}
\lesssim&\norm{\u}_{\dot{B}^{\frac {2}{p}}_{p,1}}^2\nn\\
\lesssim&\norm{\p\u}_{\dot{B}^{\frac {2}{p}}_{p,1}}^2+\norm{\q\u^\ell}_{\dot{B}^{1}_{2,1}}^2
+\norm{\q\u^h}_{\dot{B}^{\frac {2}{p}}_{p,1}}^2\nn\\
\lesssim&\norm{\p\u}_{\dot{B}^{\frac {2}{p}-1}_{p,1}}\norm{\p\u}_{\dot{B}^{\frac {2}{p}+1}_{p,1}}
+\norm{\q\u^\ell}_{\dot{B}^{0}_{2,1}}\norm{\q\u^\ell}_{\dot{B}^{2}_{2,1}}
+\norm{\q\u^h}_{\dot{B}^{\frac {2}{p}-1}_{p,1}}\norm{\q\u^h}_{\dot{B}^{\frac {2}{p}+1}_{p,1}}\nn\\
\lesssim&\e(t)\ee(t).
\end{align}
By the embedding relation in high frequency and the Young inequality, we get
\begin{align}\label{}
\norm{(\varphi \u)^h}_{\dot{B}^{\frac {2}{p}-1}_{p,1}}
\lesssim\norm{\varphi \u}_{\dot{B}^{\frac {2}{p}}_{p,1}}
\lesssim\norm{\varphi }_{\dot{B}^{\frac {2}{p}}_{p,1}}^2+\norm{ \u}_{\dot{B}^{\frac {2}{p}}_{p,1}}^2
\lesssim\e(t)\ee(t).
\end{align}
At last, we deal with each term in $\mathbf{F}(a,\u,{\varphi})$.
In view of Lemmas \ref{law}, \ref{fuhe}, there holds
\begin{align}\label{}
\norm{I(a)\nabla{\varphi}}_{\dot{B}^{\frac {2}{p}-1}_{p,1}}
\lesssim&\norm{I(a)}_{\dot{B}^{\frac {2}{p}-1}_{p,1}}\norm{\nabla{\varphi}^\ell}_{\dot{B}^{\frac {2}{p}}_{p,1}}+\norm{I(a)}_{\dot{B}^{\frac {2}{p}}_{p,1}}\norm{\nabla{\varphi}^h}_{\dot{B}^{\frac {2}{p}-1}_{p,1}}\nn\\
\lesssim&\norm{I(a)}_{\dot{B}^{\frac {2}{p}-1}_{p,1}}\norm{{\varphi}^\ell}_{\dot{B}^{2}_{2,1}}+\norm{a}_{\dot{B}^{\frac {2}{p}}_{p,1}}\norm{{\varphi^h}}_{\dot{B}^{\frac {2}{p}}_{p,1}}\nn\\
\lesssim&\norm{I(a)}_{\dot{B}^{\frac {2}{p}-1}_{p,1}}\norm{{\varphi}^\ell}_{\dot{B}^{2}_{2,1}}+(\norm{a^\ell}_{\dot{B}^{ 0}_{2,1}}+\norm{a^h}_{\dot{B}^{\frac {2}{p}}_{p,1}})\norm{{\varphi^h}}_{\dot{B}^{\frac {2}{p}}_{p,1}}
\end{align}
from which and \eqref{zijie2}, we can further get
\begin{align}\label{}
\norm{I(a)\nabla{\varphi}}_{\dot{B}^{\frac {2}{p}-1}_{p,1}}
\lesssim
(1+\e(t))\e(t)(\|\varphi\|^\ell_{\dot{B}_{2,1}^{2}}+\| {\varphi}^h\|_{\dot{B}_{p,1}^{\frac {2}{p}}}).
\end{align}
The last term in $\mathbf{F}(a,\u,{\varphi})$ can be bounded  in the same manner.
Hence, we have
\begin{align}\label{}
\norm{(I(a)(\Delta \u+\nabla\div \u))^h}_{\dot{B}^{\frac {2}{p}-1}_{p,1}}
\lesssim&\norm{I(a)}_{\dot{B}^{\frac {2}{p}}_{p,1}}\norm{\u}_{\dot{B}^{\frac {2}{p}+1}_{p,1}}\nn\\
\lesssim&\norm{a}_{\dot{B}^{\frac {2}{p}}_{p,1}}(\norm{\q\u^\ell}_{\dot{B}^{2}_{2,1}}+\norm{(\p\u^h,\q\u^h)}_{\dot{B}^{\frac {2}{p}+1}_{p,1}})\nn\\
\lesssim&\e(t)\ee(t).
\end{align}
Collecting the estimates above, we  get from \eqref{keyaguji} that
\begin{align}\label{er33}
&\norm{\varphi^h}_{\widetilde{L}^{\infty}_t(\dot{B}^{\frac {2}{p}}_{p,1})}
+\norm{(\p\u,\q\u^h)}_{\widetilde{L}^{\infty}_t(\dot{B}^{\frac {2}{p}-1}_{p,1})}
\nn\\
&\qquad+\norm{\varphi^h}_{{L}^{1}_t(\dot{B}^{\frac {2}{p}}_{p,1})}
+\norm{(\p\u,\q\u^h)}_{{L}^{1}_t(\dot{B}^{\frac {2}{p}+1}_{p,1})}\nn\\
&\quad\lesssim \norm{{\varphi}^h_{0}}_{\dot{B}^{\frac {2}{p}}_{p,1}}+
\norm{(\p\u_0,{\q\u}^h_{0})}_{\dot{B}^{\frac {2}{p}-1}_{p,1}}
+\int^t_0(1+\e(\tau))\e(\tau)\ee(\tau)\,d\tau.
\end{align}
For the term $\norm{a^h}_{\widetilde{L}^\infty_t(\dot B^{\frac  2p}_{p,1})}$, we get by a similar derivation of \eqref{han6} that
\begin{align}\label{han6aa}
\norm{a^h}_{\widetilde{L}^\infty_t(\dot B^{\frac  2p}_{p,1})}
\lesssim& \norm{a_0^h}_{\dot B^{\frac  2p}_{p,1}}+\norm{\u^h}_{L^1_t(\dot B^{\frac  2p+1}_{p,1})}+\int_0^t\norm{ \nabla\u}_{\dot B^{\frac  2p}_{p,1}}\norm{a}_{\dot B^{\frac  2p}_{p,1}}\,d\tau\nn\\
\lesssim& \norm{a_0^h}_{\dot B^{\frac  2p}_{p,1}}+\norm{(\p\u,\q\u^h)}_{L^1_t(\dot B^{\frac  2p+1}_{p,1})}\nn\\
&+\int_0^t(
\norm{\q\u^\ell}_{\dot{B}^{2}_{2,1}}+\norm{(\p\u,\q\u^h)}_{\dot{B}^{\frac {2}{p}+1}_{p,1}})(\norm{a^\ell}_{\dot{B}^{ 0}_{2,1}}+\norm{a^h}_{\dot{B}^{\frac {2}{p}}_{p,1}})\,d\tau\nn\\
\lesssim& \norm{a_0^h}_{\dot B^{\frac  2p}_{p,1}}+\norm{(\p\u,\q\u^h)}_{L^1_t(\dot B^{\frac  2p+1}_{p,1})}+\int^t_0\e(\tau)\ee(\tau)\,d\tau.
\end{align}
In the same manner, we can infer that
\begin{align}\label{han6bb}
\norm{b^h}_{\widetilde{L}^\infty_t(\dot B^{\frac  2p}_{p,1})}
\lesssim& \norm{b_0^h}_{\dot B^{\frac  2p}_{p,1}}+\norm{(\p\u,\q\u^h)}_{L^1_t(\dot B^{\frac  2p+1}_{p,1})}+\int^t_0\e(\tau)\ee(\tau)\,d\tau.
\end{align}
Multiplying by a suitable large constant on both sides of \eqref{er33} and
then pulsing  \eqref{han6aa} and \eqref{han6bb}, we can  finally get
\begin{align}\label{er3356}
&\norm{(a^h,b^h)}_{\widetilde{L}^{\infty}_t(\dot{B}^{\frac {2}{p}}_{p,1})}
+\norm{(\p\u,\q\u^h)}_{\widetilde{L}^{\infty}_t(\dot{B}^{\frac {2}{p}-1}_{p,1})}
\nn\\
&\qquad+\norm{\varphi^h}_{{L}^{1}_t(\dot{B}^{\frac {2}{p}}_{p,1})}
+\norm{(\p\u,\q\u^h)}_{{L}^{1}_t(\dot{B}^{\frac {2}{p}+1}_{p,1})}\nn\\
&\quad\lesssim \norm{({a}^h_{0},{b}^h_{0},{\varphi}^h_{0})}_{\dot{B}^{\frac {2}{p}}_{p,1}}+
\norm{(\p\u_0,{\q\u}^h_{0})}_{\dot{B}^{\frac {2}{p}-1}_{p,1}}
+\int^t_0(1+\e(\tau))\e(\tau)\ee(\tau)\,d\tau.
\end{align}

\subsection{Proof of Theorem \ref{dingli2}}
In this subsection, we shall give the proof of Theorem \ref{dingli2} by the local existence result and the continuation argument. Denote
\begin{align}\label{han22}
&\mathcal{X}(t)\stackrel{\mathrm{def}}{=}
\norm{(a^{\ell},\q\u^{\ell},{b}^{\ell})}_{\widetilde{L}^{\infty}_t(\dot{B}^{0}_{2,1})}
+\norm{(a^h,{b}^h)}_{\widetilde{L}^{\infty}_t(\dot{B}^{\frac {2}{p}}_{p,1})}
+\norm{(\p\u,\q\u^h)}_{\widetilde{L}^{\infty}_t(\dot{B}^{\frac {2}{p}-1}_{p,1})}
\notag\\
&\quad\quad\quad\quad+\norm{(\varphi^{\ell},\q\u^{\ell})}_{L^1_t(\dot{B}^{2}_{2,1})}
+\norm{\varphi^h}_{L^{1}_t(\dot{B}^{\frac {2}{p}}_{p,1})}
+\norm{(\p\u,\q\u^h)}_{L^1_t(\dot{B}^{\frac {2}{p}+1}_{p,1})},\notag\\
&\mathcal{X}_0\stackrel{\mathrm{def}}{=} \norm{(a^{\ell}_0,\q\u^{\ell}_0,{b}^{\ell}_0)}_{\dot{B}^{0}_{2,1}}
+\norm{(a^h_0,{b}^h_0)}_{\dot{B}^{\frac {2}{p}}_{p,1}}+\norm{(\p\u_0,\q\u^h_0)}_{\dot{B}^{\frac {2}{p}-1}_{p,1}}.
\end{align}
It follows from Lemmas \ref{daishu} and \ref{law}that
\begin{align}\label{tadi}
\norm{\varphi^{\ell}_0}_{\dot{B}^{0}_{2,1}}+\norm{\varphi^{h}_0}_{\dot{B}^{\frac {2}{p}}_{p,1}}
\lesssim&\norm{(a^{\ell}_0,b^{\ell}_0)}_{\dot{B}^{0}_{2,1}}+\norm{(a^{h}_0,b^{h}_0)}_{\dot{B}^{\frac {2}{p}}_{p,1}}\nn\\
&+\norm{(a^{\ell}_0,b^{\ell}_0)}_{\dot{B}^{0}_{2,1}}
(\norm{(a^{\ell}_0,b^{\ell}_0)}_{\dot{B}^{0}_{2,1}}
+\norm{(a^{h}_0,b^{h}_0)}_{\dot{B}^{\frac {2}{p}}_{p,1}})\nn\\
\lesssim&(1+\mathcal{X}_0)\mathcal{X}_0.
\end{align}
Now, summing up \eqref{ping246} and \eqref{er3356}, we get
\begin{align}\label{han23}
\mathcal{X}(t)\leq (1+\mathcal{X}_0)\mathcal{X}_0+C(\mathcal{X}(t))^2(1+C\mathcal{X}(t)).
\end{align}
Under the setting of initial data in Theorem \ref{dingli2},  there exists a positive constant $C_0$ such that
$\mathcal{X}_0\leq C_0 \epsilon$. Due to the local existence result which has been  achieved by Proposition \ref{dingli1}, there exists a positive time $T$ such that
\begin{equation}\label{re}
 \mathcal{X} (t) \leq 2 C_0\ \epsilon , \quad  \forall \; t \in [0, T].
\end{equation}
Let $T^{*}$ be the largest possible time of $T$ for what \eqref{re} holds. Now, we only need to show $T^{*} = \infty$.  By the estimate of \eqref{han23}, we can use
 a standard continuation argument to prove that $T^{*} = \infty$ provided that $\epsilon$ is small enough.  We omit the details here. This finishes the proof of Theorem \ref{dingli2}. $\quad\square$

\bigskip

\section{The proof of Theorem \ref{dingli3}}

In this section, we shall follow the method (independent of the spectral analysis) used in \cite{guoyan} and \cite{xujiang2021jde}  to get the decay rate of the solutions constructed in the previous section.
From the proof of Theorem \ref{dingli2}, we can get the following inequality (see the derivation of \eqref{ping24} and \eqref{er33} for more details):
 \begin{align}\label{sa1}
&\frac{d}{dt}(\norm{(\varphi,\u)}^\ell_{\dot{B}_{2,1}^{0}}
+\norm{\varphi}^h_{\dot{B}_{2,1}^{1}}+\norm{\u}^h_{\dot{B}_{2,1}^{0}})
+\norm{(\varphi,\u)}^\ell_{\dot{B}_{2,1}^{2}}
+\norm{\varphi}^h_{\dot{B}_{2,1}^{1}}+\norm{\u}^h_{\dot{B}_{2,1}^{2}}
\nonumber\\
&\quad\lesssim(1+\norm{(a,\u,b)}^\ell_{\dot{B}_{2,1}^{0}}
+\norm{(a,b)}^h_{\dot{B}_{2,1}^{1}}+\norm{\u}^h_{\dot{B}_{2,1}^{0}})\\
&\qquad\times(\norm{(a,\u,b)}^\ell_{\dot{B}_{2,1}^{0}}
+\norm{(a,b)}^h_{\dot{B}_{2,1}^{1}}+\norm{\u}^h_{\dot{B}_{2,1}^{0}})(\norm{(\varphi,\u)}^\ell_{\dot{B}_{2,1}^{2}}
+\norm{\varphi}^h_{\dot{B}_{2,1}^{1}}+\norm{\u}^h_{\dot{B}_{2,1}^{2}}).\nn
\end{align}
By Theorem \ref{dingli2},
\begin{align}\label{sa2}
&\norm{(a,\u,b)}^\ell_{\dot{B}_{2,1}^{0}}
+\norm{(a,b)}^h_{\dot{B}_{2,1}^{1}}+\norm{\u}^h_{\dot{B}_{2,1}^{0}}\leq c_0.
\end{align}
Inserting (\ref{sa2}) into  \eqref{sa1} yields
\begin{align}\label{sa3}
&\frac{d}{dt}(\norm{(\varphi,\u)}^\ell_{\dot{B}_{2,1}^{0}}
+\norm{\varphi}^h_{\dot{B}_{2,1}^{1}}+\norm{\u}^h_{\dot{B}_{2,1}^{0}})+\bar{c}(\norm{(\varphi,\u)}^\ell_{\dot{B}_{2,1}^{2}}
+\norm{\varphi}^h_{\dot{B}_{2,1}^{1}}+\norm{\u}^h_{\dot{B}_{2,1}^{2}})\le0.
\end{align}
In order to derive the decay estimate of the solutions given in Theorem \ref{dingli2}, we need  to get  a Lyapunov-type differential inequality from \eqref{sa3}. According to \eqref{sa2} and the embedding relation in the high frequency,
it's obvious  for any $\beta>0$ that
\begin{align}\label{sa52}
\norm{\varphi}^h_{\dot{B}_{2,1}^{1}}\ge C \big(\norm{\varphi}^h_{\dot{B}_{2,1}^{1}}\big)^{1+\beta},
\end{align}
and
\begin{align}\label{sa4}
\norm{\u}^h_{\dot{B}_{2,1}^{2}}
\ge C(\norm{\u}^h_{\dot{B}_{2,1}^{0}}
)^{1+\beta}.
\end{align}
Thus, to get the Lyapunov-type inequality of the solutions, we only need  to control the norm of $\norm{\varphi}^\ell_{\dot{B}_{2,1}^{2}}$.
This process can be obtained from  the fact that   the  solutions constructed in Theorem \ref{dingli2} can
propagate the regularity of the
initial data in Besov space with low regularity, see the following Proposition \ref{propagate}. This will ensure that one can use interpolation to get the desired Lyapunov-type inequality.

\begin{proposition}\label{propagate}
Let $(a,\u,b) $ be the  solutions constructed in Theorem \ref{dingli2} with $p=2$. Assume further that $(a_0^{\ell},\u_0^{\ell},b_0^{\ell})\in {\dot{B}^{{-\sigma}}_{2,\infty}}(\R^2)$ for some $0<\sigma\le1$.
Then there exists a constant $C_0>0$ depending on the norm of the initial data
such that for all $t\geq0$,
 \begin{eqnarray}\label{sa39-1}
\norm{(a,b,\u,\varphi)(t,\cdot)}^{\ell}_{\dot B^{{-\sigma}}_{2,\infty}}\leq C_0.
\end{eqnarray}
\end{proposition}
\begin{proof}
It follows from \eqref{ping3+1} that
\begin{align}\label{sa5}
&\frac12\frac{d}{dt}\Big(\norm{\varphi_{k}^{\ell}}^2_{L^2}/3+\norm{{d}_k^{\ell}}^2_{L^2}\Big)+2\norm{\Lambda {d}_k^{\ell}}^2_{L^2}=\big\langle(f_{1})^{\ell}_k,\varphi^{\ell}_{k}/3\big\rangle+\big\langle(f_{2})^{\ell}_k ,{d}_k^{\ell}\big\rangle
\end{align}
By performing a routine procedure, one obtains
\begin{align}\label{sa6}
\norm{(\varphi,\u)}^{\ell}_{\dot{B}^{{-\sigma}}_{2,\infty}}
\lesssim\norm{(\varphi_0,\u_0)}^{\ell}_{\dot{B}^{{-\sigma}}_{2,\infty}}
+\int_0^t\norm{(f_1,f_2)}^{\ell}_{\dot{B}^{{-\sigma}}_{2,\infty}}\,d\tau.
\end{align}
To control $\norm{a}^{\ell}_{\dot{B}^{{-\sigma}}_{2,\infty}}$, we first
get by taking $\ddj$ on  both hand side of \eqref{ggug12} that
\begin{align}\label{sa7}
\partial_t\ddj \delta+\u\cdot\nabla \ddj \delta+[\ddj,\u\cdot\nabla] \delta =\ddj f_3
\end{align}
with $f_3\stackrel{\mathrm{def}}{=}\delta\,\div \u+\varphi\,\div \u$.

Then, we get by a similar derivation of \eqref{you2} that
\begin{align}\label{sa8}
\|\delta^{\ell}\|_{\dot{B}^{-\sigma}_{2,\infty}}
\lesssim&\| \delta_0^\ell\|_{\dot{B}^{-\sigma}_{2,\infty}} +\int_0^t\| \div\u\|_{L^\infty}\| \delta^\ell\|_{\dot{B}^{-\sigma}_{2,\infty}}\,d\tau\nn\\
&+\int_0^t\| \nabla\u\|_{\dot{B}^{1}_{2,1}}\| \delta\|_{\dot{B}^{-\sigma}_{2,\infty}}\,d\tau+\int_0^t\| f_3\|_{\dot{B}^{-\sigma}_{2,\infty}}\,d\tau
\end{align}
in which we have used the  Lemma 2.100 of \cite{bcd} to  deal withe the commutator.

With the aid of the embedding relation ${\dot B^{1}_{2,1}}(\R^2)\hookrightarrow L^\infty(\R^2)$ and Lemma \ref{jiaohuanzi}, we  have
\begin{align}\label{sa9}
\| \div\u\|_{L^\infty}\| \delta^\ell\|_{\dot{B}^{-\sigma}_{2,\infty}}+\| \nabla\u\|_{\dot{B}^{1}_{2,1}}\| \delta\|_{\dot{B}^{-\sigma}_{2,\infty}}
\lesssim&\| \nabla\u\|_{\dot{B}^{1}_{2,1}}(\| \delta^\ell\|_{\dot{B}^{-\sigma}_{2,\infty}}+\| \delta^h\|_{\dot{B}^{1}_{2,1}})\nn\\
\lesssim&\|\u\|_{\dot{B}^{2}_{2,1}}(\| \delta^\ell\|_{\dot{B}^{-\sigma}_{2,\infty}}+\| \delta^h\|_{\dot{B}^{1}_{2,1}})\nn\\
\lesssim&\|\u\|_{\dot{B}^{2}_{2,1}}(\| (a^\ell,\varphi^\ell)\|_{\dot{B}^{-\sigma}_{2,\infty}}+\| (a^h,\varphi^h)\|_{\dot{B}^{1}_{2,1}}).
\end{align}
For  the las term in \eqref{sa8}, we   use Lemma \ref{guangyidaishu} directly to get
\begin{align}\label{sa10}
\| f_3\|_{\dot{B}^{-\sigma}_{2,\infty}}
\lesssim&\| \nabla\u\|_{\dot{B}^{1}_{2,1}}\| (\delta,\varphi)\|_{\dot{B}^{-\sigma}_{2,\infty}}\nn\\
\lesssim&\|\u\|_{\dot{B}^{2}_{2,1}}(\| (a^\ell,\varphi^\ell)\|_{\dot{B}^{-\sigma}_{2,\infty}}+\| (a^h,\varphi^h)\|_{\dot{B}^{1}_{2,1}}).
\end{align}
Inserting \eqref{sa9} and \eqref{sa10} into \eqref{sa8} yields
\begin{align}\label{sa11}
\|\delta^{\ell}\|_{\dot{B}^{-\sigma}_{2,\infty}}
\lesssim&\| \delta_0^\ell\|_{\dot{B}^{-\sigma}_{2,\infty}} +\int_0^t\|\u\|_{\dot{B}^{2}_{2,1}}(\| (a^\ell,\varphi^\ell)\|_{\dot{B}^{-\sigma}_{2,\infty}}+\| (a^h,\varphi^h)\|_{\dot{B}^{1}_{2,1}})\,d\tau
\end{align}
from which and \eqref{sa6}, we obtain
\begin{align}\label{sa12}
\norm{(a,\varphi,\u)}^{\ell}_{\dot{B}^{{-\sigma}}_{2,\infty}}
\lesssim&\norm{(a_0,\varphi_0,\u_0)}^{\ell}_{\dot{B}^{{-\sigma}}_{2,\infty}}
+\int_0^t\norm{(f_1,f_2)}^{\ell}_{\dot{B}^{{-\sigma}}_{2,\infty}}\,d\tau\nn\\
&+\int_0^t\|\u\|_{\dot{B}^{2}_{2,1}}(\| (a^\ell,\varphi^\ell)\|_{\dot{B}^{-\sigma}_{2,\infty}}+\| (a^h,\varphi^h)\|_{\dot{B}^{1}_{2,1}})\,d\tau.
\end{align}
In the same manner, we also can get
\begin{align}\label{sa13}
\norm{(b,\varphi,\u)}^{\ell}_{\dot{B}^{{-\sigma}}_{2,\infty}}
\lesssim&\norm{(b_0,\varphi_0,\u_0)}^{\ell}_{\dot{B}^{{-\sigma}}_{2,\infty}}
+\int_0^t\norm{(f_1,f_2)}^{\ell}_{\dot{B}^{{-\sigma}}_{2,\infty}}\,d\tau\nn\\
&+\int_0^t\|\u\|_{\dot{B}^{2}_{2,1}}(\| (b^\ell,\varphi^\ell)\|_{\dot{B}^{-\sigma}_{2,\infty}}+\| (b^h,\varphi^h)\|_{\dot{B}^{1}_{2,1}})\,d\tau
\end{align}
which combines with \eqref{sa12} leads to
\begin{align}\label{sa14}
\norm{(a,b,\varphi,\u)}^{\ell}_{\dot{B}^{{-\sigma}}_{2,\infty}}
\lesssim&\norm{(a_0,b_0,\varphi_0,\u_0)}^{\ell}_{\dot{B}^{{-\sigma}}_{2,\infty}}
+\int_0^t\norm{(f_1,f_2)}^{\ell}_{\dot{B}^{{-\sigma}}_{2,\infty}}\,d\tau\nn\\
&+\int_0^t\|\u\|_{\dot{B}^{2}_{2,1}}(\| (a^\ell,b^\ell,\varphi^\ell)\|_{\dot{B}^{-\sigma}_{2,\infty}}+\| (a^h,b^h,\varphi^h)\|_{\dot{B}^{1}_{2,1}})\,d\tau.
\end{align}

To bound the nonlinear terms in $f_1, f_2$, we need  the following two crucial
estimates which can be obtained from Lemma \ref{guangyidaishu} directly.

\begin{align}\label{key}
\bullet\qquad\|fg\|_{\dot{B}_{2,\infty}^{{-{s}}}}^\ell
\lesssim&\|f\|_{\dot{B}_{2,1}^{1}}
\|g\|_{\dot{B}_{2,\infty}^{{-{s}}}},\quad -1<{{s}}\le1,
\end{align}
\begin{align}\label{key3}
\bullet\qquad\|fg\|_{\dot{B}_{2,\infty}^{{-{s}}}}^\ell
\lesssim&\|f\|_{\dot{B}_{2,1}^{0}}
\|g\|_{\dot{B}_{2,\infty}^{{-{s}}+1}},\quad 0<{{s}}\le1.
\end{align}
To simplify the notation, we set
\begin{align*}
&\mathcal{D}_\infty(t)\stackrel{\mathrm{def}}{=}\norm{(a,\u,b)}^\ell_{\dot{B}_{2,1}^{0}}
+\norm{(a,b)}^h_{\dot{B}_{2,1}^{1}}+\norm{\u}^h_{\dot{B}_{2,1}^{0}},
\nonumber\\
&\mathcal{D}_1(t)\stackrel{\mathrm{def}}{=}\norm{(\varphi,\u)}^\ell_{\dot{B}_{2,1}^{2}}
+\norm{\varphi}^h_{\dot{B}_{2,1}^{1}}+\norm{\u}^h_{\dot{B}_{2,1}^{2}}
.
\end{align*}
From \eqref{key}, one has
\begin{align}\label{asa15}
&\norm{\u\cdot\nabla \varphi^{\ell}}^{\ell}_{\dot{B}^{{-\sigma}}_{2,\infty}}+\norm{\varphi\,\mathrm{div}\,\u^{\ell}}^{\ell}_{\dot{B}^{{-\sigma}}_{2,\infty}}
\nonumber\\&\quad
\lesssim\norm{\u^\ell}_{\dot{B}^{-\sigma}_{2,\infty}} \norm{\nabla \varphi^\ell}_{\dot{B}^1_{2,1}}
+\norm{\u}^h_{\dot B^1_{2,1}}\norm{\nabla \varphi^{\ell}}_{\dot B^{-\sigma}_{2,\infty}}
+\norm{\varphi^\ell}_{\dot B^{-\sigma}_{2,\infty}}\norm{\mathrm{div}\,\u^\ell}_{\dot B^1_{2,1}}
\nonumber \\
&\qquad +\norm{\varphi^h}_{\dot B^1_{2,1}}\norm{\mathrm{div}\,\u^{\ell}}_{\dot B^{-\sigma}_{2,\infty}}\nonumber\\
&\quad\lesssim\norm{\varphi^{\ell}}_{\dot{B}^{2}_{2,1}}\norm{\u^{\ell}}_{\dot{B}^{{-\sigma}}_{2,\infty}}+\norm{\u}^{h}_{\dot{B}^{2}_{2,1}}\norm{\varphi^{\ell}}_{\dot{B}^{{-\sigma}}_{2,\infty}}
+\norm{\u}^{\ell}_{\dot{B}^{2}_{2,1}}\norm{\varphi}^{\ell}_{\dot{B}^{{-\sigma}}_{2,\infty}}+\norm{\varphi}^{h}_{\dot{B}^{1}_{2,1}}\norm{\u}^{\ell}_{\dot{B}^{{-\sigma}}_{2,\infty}}\nonumber\\
&\quad\lesssim\er\norm{(\varphi^{\ell},\u^{\ell})}_{\dot{B}^{{-\sigma}}_{2,\infty}}.
\end{align}
Thanks to \eqref{key3}, we have
\begin{align}\label{asa16}
&\norm{\u\cdot\nabla \varphi^h}_{\dot B^{{-\sigma}}_{2,\infty}}^\ell+\norm{\varphi\mathrm{div}\u^{h}}_{\dot B^{{-\sigma}}_{2,\infty}}^\ell
\nonumber\\
&\quad\lesssim \norm{\u}^{\ell}_{\dot B^{{-\sigma}+1}_{2,\infty}}\norm{\nabla \varphi}^h_{\dot B^{0}_{2,1}}+\norm{\nabla \varphi}^h_{\dot B^{0}_{2,1}}\norm{\u}^h_{\dot B^{1}_{2,1}}
 +(\norm{\varphi}^\ell_{\dot B^{0}_{2,1}}+\norm{\varphi}^h_{\dot B^{0}_{2,1}})\norm{\mathrm{div}\u}^h_{\dot B^{1}_{2,1}}\nonumber
\\
&\quad\lesssim \norm{\varphi}^h_{\dot B^{1}_{2,1}}\norm{\u}^{\ell}_{\dot B^{{-\sigma}}_{2,\infty}}+(\norm{\varphi}^{\ell}_{\dot B^{0}_{2,1}}+\norm{\varphi}^h_{\dot B^{1}_{2,1}})\norm{\u}^h_{\dot B^{2}_{2,1}}\nn\\
&\quad\lesssim\er\norm{\u^{\ell}}_{\dot{B}^{{-\sigma}}_{2,\infty}}+\mathcal{D}_\infty(t)\er
\end{align}
 which, together with \eqref{asa15},
gives
\begin{align}\label{asa18}
\norm{f_1}^{\ell}_{\dot{B}^{{-\sigma}}_{2,\infty}}
\lesssim\er\norm{(\varphi^{\ell},\u^{\ell})}_{\dot{B}^{{-\sigma}}_{2,\infty}}+\mathcal{D}_\infty(t)\er.
\end{align}
Next, we  bound the terms in $f_2$. The estimate of $\u\cdot\nabla \u$ follows from essentially the same procedures as $\mathrm{div}(\varphi\u)$ so that
\begin{align}\label{asa19}
\norm{\u\cdot \nabla \u}^{\ell}_{\dot{B}^{{-\sigma}}_{2,\infty}}\lesssim&\norm{\u\cdot \nabla \u^{\ell}}^{\ell}_{\dot{B}^{{-\sigma}}_{2,\infty}}+\norm{\u\cdot \nabla \u^{h}}_{\dot B^{{-\sigma}}_{2,\infty}}^\ell \nn\\
\lesssim &\big(\norm{\u}^{\ell}_{\dot{B}^{2}_{2,1}}+\norm{\u}^{h}_{\dot{B}^{2}_{2,1}}\big)\norm{\u}^{\ell}_{\dot{B}^{{-\sigma}}_{2,\infty}}+\big(\norm{\u}^{\ell}_{\dot B^{0}_{2,1}}+\norm{\u}^{h}_{\dot B^{0}_{2,1}}\big)\norm{\u}^{h}_{\dot B^{2}_{2,1}}\nn\\
\lesssim &\er\norm{\u^{\ell}}_{\dot{B}^{{-\sigma}}_{2,\infty}}+\mathcal{D}_\infty(t)\er.
\end{align}
For the term $I(a)(\Delta \u+\nabla\div \u)$, it follows from  \eqref{key3} that
\begin{align}\label{asa20}
\|I(a)(\Delta \u+\nabla\div \u)\|_{\dot{B}_{2,\infty}^{{-\sigma}}}^\ell
\lesssim&\|\Delta \u+\nabla\div \u\|_{\dot{B}_{2,1}^{0}}
\|I(a)\|_{\dot{B}_{2,\infty}^{{-\sigma}+1}}
\nn\\
\lesssim&\| \u\|_{\dot{B}_{2,1}^{2}}
(\|(I(a))^\ell\|_{\dot{B}_{2,\infty}^{{-\sigma}+1}}+\|(I(a))^h\|_{\dot{B}_{2,\infty}^{{-\sigma}+1}})\nn\\
\lesssim&(\| \u^\ell\|_{\dot{B}_{2,1}^{2}}+\| \u^h\|_{\dot{B}_{2,1}^{2}})
(\|(I(a))^\ell\|_{\dot{B}_{2,1}^{0}}+\|(I(a))^h\|_{\dot{B}_{2,1}^{1}}).
\end{align}
In view of  the previous estimates \eqref{zijie1} and \eqref{zijie2}, there holds
\begin{align}\label{asa21}
\|(I(a))^\ell\|_{\dot{B}_{2,1}^{0}}+\|(I(a))^h\|_{\dot{B}_{2,1}^{1}}
\lesssim(1+\es)\es.
\end{align}
Hence, we infer from \eqref{asa20} that
\begin{align}\label{asa22}
\|I(a)(\Delta \u+\nabla\div \u)\|_{\dot{B}_{2,\infty}^{{-\sigma}}}^\ell
\lesssim&(1+\es)\es\er.
\end{align}
For the last term in $f_2$, we exploit \eqref{key} and \eqref{key3}  to get
\begin{align}\label{asa23}
\|I(a)\nabla\varphi\|_{\dot{B}_{2,\infty}^{{-\sigma}}}^\ell
\lesssim&\|I(a)\|_{\dot{B}_{2,\infty}^{{-\sigma}}}\|\nabla\varphi^\ell\|_{\dot{B}_{2,1}^{1}}
+\|I(a)\|_{\dot{B}_{2,\infty}^{{-\sigma+1}}}\|\nabla\varphi^h\|_{\dot{B}_{2,1}^{0}}\nn\\
\lesssim&\|I(a)\|_{\dot{B}_{2,\infty}^{{-\sigma}}}\|\varphi^\ell\|_{\dot{B}_{2,1}^{2}}
+\|I(a)\|_{\dot{B}_{2,\infty}^{{-\sigma+1}}}\|\varphi^h\|_{\dot{B}_{2,1}^{1}}.
\end{align}
Keeping in mind that  $I(a)=a-aI(a),$ we use Lemma \ref{guangyidaishu} and \eqref{zijie1}, \eqref{zijie2} to write
\begin{align}\label{asa24}
\|I(a)\|_{\dot{B}_{2,\infty}^{{-\sigma}}}
\lesssim&\|a\|_{\dot{B}_{2,\infty}^{{-\sigma}}}+\|aI(a)\|_{\dot{B}_{2,\infty}^{{-\sigma}}}\nn\\
\lesssim&\|a\|_{\dot{B}_{2,\infty}^{{-\sigma}}}+\|a\|_{\dot{B}_{2,\infty}^{{-\sigma}}}\|I(a)\|_{\dot{B}_{2,1}^{1}} \nn\\
\lesssim&(1+\|(I(a))^\ell\|_{\dot{B}_{2,1}^{0}}+\|(I(a))^h\|_{\dot{B}_{2,1}^{1 }})(\|a^\ell\|_{\dot{B}_{2,\infty}^{{-\sigma}}}+\|a^h\|_{\dot{B}_{2,1}^{1}})\nn\\
\lesssim&(1+(\es)^2)\|a^\ell\|_{\dot{B}_{2,\infty}^{{-\sigma}}}+(1+(\es)^2)\es.
\end{align}
Similar to previous estimate, one has
\begin{align}\label{asa25}
\|I(a)\|_{\dot{B}_{2,\infty}^{{-\sigma+1}}}\lesssim\|(I(a))^\ell\|_{\dot{B}_{2,1}^{0}}+\|(I(a))^h\|_{\dot{B}_{p,1}^{\frac {2}{p}}}\lesssim(1+\es)\es.
\end{align}
Taking \eqref{asa24} and \eqref{asa25} into \eqref{asa23} gives
\begin{align}\label{asa26}
\|I(a)\nabla\varphi\|_{\dot{B}_{2,\infty}^{{-\sigma}}}^\ell
\lesssim&(1+(\es)^2)\er\|a^\ell\|_{\dot{B}_{2,\infty}^{{-\sigma}}}+(1+(\es)^2)\es\er.
\end{align}
Collecting the  estimates  \eqref{asa19}, \eqref{asa22},  and \eqref{asa26}, we obtain
\begin{align}\label{asa27}
\|f_2\|^{\ell}_{\dot{B}^{{-\sigma}}_{2,\infty}}
\lesssim&(1+(\es)^2)\er\|(a^\ell,\u^\ell)\|_{\dot{B}_{2,\infty}^{{-\sigma}}}+(1+\mathcal{D}_\infty(t))\mathcal{D}_\infty(t)\er.
\end{align}
Plugging   \eqref{asa18} and \eqref{asa27} into \eqref{sa14}, we finally arrive at
\begin{align}\label{asa28}
\norm{(a,b,\varphi,\u)}^{\ell}_{\dot{B}^{{-\sigma}}_{2,\infty}}
\lesssim&\norm{(a_0,b_0,\varphi_0,\u_0)}^{\ell}_{\dot{B}^{{-\sigma}}_{2,\infty}}
+\int_0^t(1+\mathcal{D}_\infty(\tau))\mathcal{D}_\infty(\tau)\mathcal{D}_1(\tau) \,d\tau\nonumber\\
&\quad\quad+\int_0^t
(1+(\mathcal{D}_\infty(\tau))^2)\mathcal{D}_1(\tau)\|(a,b,\varphi,\u)\|^{\ell}_{\dot{B}^{{-\sigma}}_{2,\infty}}\,d\tau.
\end{align}
Noticing that the definition of $\varphi_0$ in \eqref{mm3}, it is easy to deduce from Lemma \ref{guangyidaishu} that
\begin{align*}
\norm{\varphi_0}^{\ell}_{\dot{B}^{{-\sigma}}_{2,\infty}}
\lesssim&\norm{(a_0,b_0)}^{\ell}_{\dot{B}^{{-\sigma}}_{2,\infty}}+\norm{(a_0,b_0)}_{\dot{B}^{{-\sigma}}_{2,\infty}}
\norm{(a_0,b_0)}_{\dot{B}^{{1}}_{2,1}}\nonumber\\
\lesssim&\norm{(a_0,b_0)}^{\ell}_{\dot{B}^{{-\sigma}}_{2,\infty}}+(\norm{(a_0,b_0)}^{\ell}_{\dot{B}^{{-\sigma}}_{2,\infty}}+\norm{(a_0^h,b_0^h)}_{\dot{B}^{{1}}_{2,1}})
(\norm{(a_0^\ell,b_0^\ell)}_{\dot{B}^{{0}}_{2,1}}+\norm{(a_0^h,b_0^h)}_{\dot{B}^{{1}}_{2,1}}).
\end{align*}
Consequently, one can employ nonlinear generalisations of the Gronwall's inequality
to get
\begin{eqnarray}\label{sa39}
\norm{(a,b,\u,\varphi)(t,\cdot)}^{\ell}_{\dot B^{{-\sigma}}_{2,\infty}}\leq C_0
\end{eqnarray}
for all $t\geq0$, where $C_0>0$ depends on the norm of  $a_0,b_0,\u_0$.
This completes the proof of Proposition \ref{propagate}.
\end{proof}

\medskip
Now, we prove the  Lyapunov-type  inequality from \eqref{sa3}.
For any $0<{\sigma}\le1,$
it follows from an interpolation inequality that
\begin{align*}
\norm{(\varphi,\u)}^\ell_{\dot{B}_{2,1}^{0}}
\le&C\big(\norm{(\varphi,\u)}_{\dot{B}_{2,\infty}^{{-\sigma}}}\big)^{\alpha_1}\big(\norm{(\varphi,\u)}^\ell_{\dot{B}_{2,1}^{2}}\big)^{1-\alpha_1},
\quad \alpha_1=\frac{2}{2+{\sigma}}\in(0,1),
\end{align*}
which, together with Proposition \ref{propagate},  implies
\begin{align}\label{sa51}
\norm{(\varphi,\u)}^\ell_{\dot{B}_{2,1}^{2}}\ge  c_0\big(\norm{(\varphi,\u)}^\ell_{\dot{B}_{2,1}^{0}}\big)^{\frac{1}{1-\alpha_1}}.
\end{align}
Taking $\beta=1+\alpha_1>0$ in
\eqref{sa52} and \eqref{sa4} and combining with   \eqref{sa51}, we deduce from \eqref{sa3}
that
\begin{align}\label{sa58}
&\frac{d}{dt}(\norm{(\varphi,\u)}^\ell_{\dot{B}_{2,1}^{0}}
+\norm{\varphi}^h_{\dot{B}_{2,1}^{1}}+\norm{\u}^h_{\dot{B}_{2,1}^{0}})\nn\\
&\quad+\widetilde{c}_0(\norm{(\varphi,\u)}^\ell_{\dot{B}_{2,1}^{0}}
+\norm{\varphi}^h_{\dot{B}_{2,1}^{1}}+\norm{\u}^h_{\dot{B}_{2,1}^{0}})^{1+\frac{2}{{\sigma}}}\le0.
\end{align}
Solving this differential inequality directly, we obtain
\begin{align}\label{sa59}
\norm{(\varphi,\u)}^\ell_{\dot{B}_{2,1}^{0}}
+\norm{\varphi}^h_{\dot{B}_{2,1}^{1}}+\norm{\u}^h_{\dot{B}_{2,1}^{0}}
\le C(1+t)^{-\frac{{\sigma}}{2}}.
\end{align}
For any ${-\sigma}<\gamma_1<0,$ by the interpolation inequality, we have
\begin{align*}
\norm{(\varphi,\u)}^\ell_{\dot{B}_{2,1}^{\gamma_1}}
\le&C\big(\norm{(\varphi,\u)}^\ell_{\dot{B}_{2,\infty}^{{-\sigma}}}\big)^{\alpha_2} \big(\norm{(\varphi,\u)}^\ell_{\dot{B}_{2,1}^{0}}\big)^{1-\alpha_2},\quad \alpha_2=-\frac{\gamma_1}{{\sigma}}\in (0,1),
\end{align*}
which, together  with  Proposition \ref{propagate}, gives
\begin{align}\label{sa60}
\norm{(\varphi,\u)}^\ell_{\dot{B}_{2,1}^{\gamma_1}}
\le C(1+t)^{\frac{{-\sigma}(1-\alpha_2)}{2}}=C(1+t)^{-\frac{\gamma_1+{\sigma}}{2}}.
\end{align}
In  the light of
${-\sigma}<\gamma_1<0,$
 we see that
$$\norm{(\varphi^h,\u^h)}_{\dot{B}_{2,1}^{\gamma_1}}\le C(\norm{\varphi}^h_{\dot{B}_{2,1}^{1}}+\norm{\u}^h_{\dot{B}_{2,1}^{0 }})\le C(1+t)^{-\frac{{\sigma}}{2}},
$$
which, together with \eqref{sa60}, yields
\begin{align*}
\norm{(\varphi,\u)}_{\dot{B}_{2,1}^{\gamma_1}}
\le&C(\norm{(\varphi,\u)}^\ell_{\dot{B}_{2,1}^{\gamma_1}}+\norm{(\varphi,\u)}^h_{\dot{B}_{2,1}^{\gamma_1}})\nonumber\\
\le& C(1+t)^{-\frac{\gamma_1{+\sigma}}{2}}+C(1+t)^{-\frac{{\sigma}}{2}}\nonumber\\
\le& C(1+t)^{-\frac{\gamma_1{+\sigma}}{2}}.
\end{align*}
Hence,
thanks to the embedding relation
$\dot{B}^{0}_{2,1}(\R^2)\hookrightarrow L^2(\R^2)$, one infers that
\begin{align*}
\norm{\Lambda^{\gamma_1} (\varphi,\u)}_{L^2}
\le& C(1+t)^{-\frac{\gamma_1{+\sigma}}{2}}.
\end{align*}
This completes the proof of Theorem \ref{dingli3}. $\quad \square$

\bigskip
\section*{Acknowledgments}
Dong was partially supported by the National Natural Science Foundation of China under grant 11871346, the NSF of Guangdong Province under grant 2020A1515010530, NSF of Shenzhen City (Nos.JCYJ20180305125554234, 20200805101524001).
Wu was partially supported by the National Science Foundation of the USA under grant DMS 2104682, the Simons Foundation grant (Award number 708968) and the AT\&T Foundation at Oklahoma State University. Zhai was partially supported by the National Natural Science Foundation of China under grant11601533, the Guangdong Provincial Natural Science Foundation under grant 2022A1515011977 and the Science and Technology Program of Shenzhen under grant  20200806104726001.

  \bigskip

\noindent{\bf Data Availability Statement}: Our manuscript has no associated data.

\end{document}